\documentclass{amsart}
\usepackage{amsmath}
\usepackage{paralist}
\usepackage{url}
\usepackage{mathrsfs}
\usepackage{color}
\usepackage{graphicx}
\usepackage{caption}
\usepackage{subcaption}
\usepackage{multirow}
\usepackage{mathtools}
\usepackage{amssymb}
\usepackage{booktabs}
\usepackage{makecell}
\usepackage{pgfplots}
\pgfplotsset{compat=newest}
\usepackage{mathrsfs}
\usetikzlibrary{arrows}
\usepackage{xspace}
\usepackage{comment}
\usepackage{arydshln}
\usepackage{caption}

\newcommand{\re}{\mathbb{R}}

\newcommand{\N}{\mathbb{N}}

\newcommand{\lmd}{\lambda}

\def\rank{\mbox{rank}}

\newcommand{\reff}[1]{(\ref{#1})}

\newcommand{\mc}[1]{\mathcal{#1}}

\newcommand{\supp}[1]{\mbox{supp}(#1)}
\newcommand{\st}{\mathit{s.t.}}

\newcommand{\qmod}[1]{\mbox{QM}[#1]}

\newcommand{\RN}[1]{%
  \textup{\uppercase\expandafter{\romannumeral#1}}%
}

\newcommand{\bdes}{\begin{description}}
	\newcommand{\edes}{\end{description}}
\newcommand{\bal}{\begin{align}}
\newcommand{\eal}{\end{align}}
\newcommand{\bnum}{\begin{enumerate}}
	\newcommand{\enum}{\end{enumerate}}
\newcommand{\bit}{\begin{itemize}}
	\newcommand{\eit}{\end{itemize}}
\newcommand{\bea}{\begin{eqnarray}}
\newcommand{\eea}{\end{eqnarray}}
\newcommand{\be}{\begin{equation}}
\newcommand{\ee}{\end{equation}}
\newcommand{\baray}{\begin{array}}
	\newcommand{\earay}{\end{array}}
\newcommand{\bsry}{\begin{subarray}}
	\newcommand{\esry}{\end{subarray}}
\newcommand{\bca}{\begin{cases}}
	\newcommand{\eca}{\end{cases}}
\newcommand{\bcen}{\begin{center}}
	\newcommand{\ecen}{\end{center}}
\newcommand{\bbm}{\begin{bmatrix}}
	\newcommand{\ebm}{\end{bmatrix}}
\newcommand{\bmx}{\begin{matrix}}
	\newcommand{\emx}{\end{matrix}}
\newcommand{\bpm}{\begin{bmatrix}}
	\newcommand{\epm}{\end{bmatrix}}
\newcommand{\btab}{\begin{tabular}}
	\newcommand{\etab}{\end{tabular}}

\newcommand{\MATLAB}{\textsc{Matlab}\xspace}

\newtheorem{theorem}{Theorem}[section]

\newtheorem{prop}[theorem]{Proposition}

\newtheorem{cor}[theorem]{Corollary}

\theoremstyle{definition}
\newtheorem{example}[theorem]{Example}

\newcommand{\ddd}{,\ldots,}

\numberwithin{equation}{section}

\pgfplotsset{every axis/.append style={tick label style={/pgf/number format/fixed},font=\scriptsize,ylabel near ticks,xlabel near ticks,grid=major}}

\pgfmathdeclarefunction{sumexp}{3}{%
  \begingroup%
  \pgfkeys{/pgf/fpu}
  \pgfmathsetmacro{\myx}{#1}%
  \pgfmathtruncatemacro{\myxmin}{#2}%
  \pgfmathtruncatemacro{\myxmax}{#3}%
  \pgfmathsetmacro{\mysum}{0}%
  \pgfplotsforeachungrouped\XX in {\myxmin,...,\myxmax}%
    {\pgfmathsetmacro{\mysum}{\mysum+exp(\XX)}}%
  \pgfmathparse{\mysum+exp(#1)}%
  \pgfmathfloattofixed\pgfmathresult
  \pgfmathsmuggle\pgfmathresult\endgroup%
}%

\begin{document}

\title[Log-Polynomial Optimization]
{Log-Polynomial Optimization}

\author[Jiyoung Choi]{Jiyoung~Choi}
\author[Jiawang Nie]{Jiawang~Nie}
\author[Xindong Tang]{Xindong~Tang}
\author[Suhan Zhong]{Suhan~Zhong}

\address{Jiyoung Choi, Jiawang Nie, Department of Mathematics, 
	University of California San Diego, 9500 Gilman Drive, La Jolla, CA, USA, 92093.}
\email{jichoi@ucsd.edu, njw@math.ucsd.edu}

\address{Xindong Tang, Department of Mathematics,
Hong Kong Baptist University,
Kowloon Tong, Kowloon, Hong Kong.}
\email{xdtang@hkbu.edu.hk}

\address{Suhan Zhong, School of Mathematical Sciences,
Shanghai Jiao Tong University,
Shanghai, China, 200240.}
\email{suzhong@sjtu.edu.cn}

\begin{abstract}
We study an optimization problem in which the objective is 
given as a sum of logarithmic-polynomial functions.
This formulation is motivated by statistical estimation principles 
such as maximum likelihood estimation, 
and by loss functions including cross-entropy and Kullback-Leibler divergence.
We propose a hierarchy of moment relaxations based on the truncated 
$K$-moment problems to solve log-polynomial optimization.
We provide sufficient conditions for the hierarchy to be tight 
and introduce a numerical method to extract the global optimizers 
when the tightness is achieved.
In addition, we modify relaxations with optimality conditions 
to better fit log-polynomial optimization with convenient 
Lagrange multipliers expressions.
Various applications and numerical experiments are 
presented to show the efficiency of our method.
\end{abstract}

\maketitle

\section{Introduction}

Consider the log-polynomial optimization
\begin{equation}\label{eq:problem}
\left\{\begin{array}{cl}
\max\limits_{x\in\re^n} &
 f(x) = \sum\limits_{i=1}^m a_i\log p_i(x)  \\
\st & c_j(x) = 0 \,\, (j \in \mathcal{E}),\\
 & c_j(x) \geq 0 \,\, (j \in \mathcal{I}),
\end{array}\right.
\end{equation}
where $a_1,\ldots, a_m$ are positive scalars, 
$\mathcal{E}$ and $\mathcal{I}$ are two finite disjoint label sets
and $p_1,\ldots, p_m, c_j$ are polynomials for all $j\in \mc{E}\cup J$. 
Denote \(c_{\mc{E}} = (c_j)_{j\in\mc{E}}\) and \(c_{\mc{I}} = (c_j)_{j\in\mc{I}}\).
The feasible set of \eqref{eq:problem} can be written as
\begin{equation}\label{eq:K}
	K  \coloneqq  \{ x\in\re^n : c_{\mc{E}}(x) = 0,\, c_{\mc{I}}(x) \ge  0\}.
\end{equation}
To ensure the objective function in \eqref{eq:problem} is well defined, 
we assume $p_i\ge 0$ over \(K\) and there exists a point $\hat{x}\in K$
such that $p_i(\hat{x})>0$ for all $i\in \mc{E}\cup \mc{I}$ throughout this paper.
For convenience, denote 
\[
p^a(x) \coloneqq p_1(x)^{a_1}\cdots p_m(x)^{a_m} = e^{f(x)}.
\]
The log-polynomial optimization \eqref{eq:problem}
is typically nonconvex.
If all $a_i$ are positive integers, 
then $p^a$ is a polynomial.
In this case, the problem \eqref{eq:problem} 
is equivalent to maximizing $p^a$ over $K$.
While the polynomial optimization reformulation 
can be globally solved via Moment-SOS relaxations \cite{lasserre2001,MPO}, 
the computational cost increases significantly when the exponents $ a_i $ are large. 
If any $a_i$ is not an integer, the problem \eqref{eq:problem}
usually cannot be expressed as a polynomial optimization problem.
Consequently, when $a_i$ for all $i$ are general positive scalars,
it is difficult to compute a useful upper bound 
or certify the global optimality of \eqref{eq:problem}.

Log-polynomials has many applications. 
In combinatorial optimization, 
log-concave polynomials serve as an important tool 
to study counting problems \cite{Anari24}. 
In statistical learning, many common loss functions,
such as the log-likelihood function and cross-entropy, 
are defined with logarithmic terms.
These functions are widely used in density estimation 
\cite{Carpenter2018,Dempster1977,Fisher1922,Vinayak2019,Silverman1986} and
classification \cite{Bishop2006,Goodfellow2016,Goodfellow2014,Mannor2005,Mao2023,Zhang2018} tasks,
where parameters to be learned are often subject to
structural constraints \cite{Gusmao2024,Hansen2024}
such as non-negativity, normalization, or sparsity-inducing regularization
\cite{Banerjee2008,Karanam2025,Lauritzen2019,Rayas2022,Srivastava2024}.
Suppose the selected models are parametrized by polynomial functions.
Then the learning process reduces to solving log-polynomial optimization.
We illustrate this framework via the following motivating applications.

{\it Maximum likelihood estimation (MLE)} is a core problem in statistics.
Consider a sample sequence $\mc{D}$ generated from a polynomial 
model $p(\xi; x)$, where the parameter $x$ is constrained in a feasible set $K$.
Let $S = \supp{\mc{D}} = \{\xi^{(1)}, \ldots, \xi^{(m)}\}$ 
denote the support set of $\mc{D}$.
We use $a_i$ to count the number of samples in $\mc{D}$ that equal $\xi^{(i)}$
for $i=1,\ldots, m$.
The {\em likelihood function} associated with $\mc{D}$ 
and the corresponding {\em log-likelihood function} are respectively 
defined as
\[
L(\mc{D};x) \coloneqq \prod\limits_{i=1}^m p(x;\xi^{(i)})^{a_i},
\quad 
\log L(\mc{D}; x) = \sum\limits_{i=1}^m a_i\log p(x;\xi^{(i)}).
\]
The MLE aims to find parameter values that maximize the likelihood function,
which is equivalent to maximizing the log-likelihood.

{\it Cross-entropy} and {\it Kullback-Leibler (KL) divergence} 
are two common loss functions in classification tasks.
Consider the discrete random variable $\xi$ supported on 
$S = \{\xi^{(1)}, \ldots, \xi^{(m)}\}$, 
whose true distribution $P$ has $p(\xi)$ to be its {\em probability mass function} (PMF). 
Let $Q(x)$ be the predicted distribution with a polynomial PMF $q(\xi; x)$.
The {\em cross-entropy} between $P$ and $Q(x)$ is
\[
H(P,Q(x)) \coloneqq -\sum\limits_{i=1}^m p(\xi^{(i)})\log q(\xi^{(i)}; x).
\]
The KL divergence from $P$ to $Q(x)$ is 
\[
D_{\operatorname{KL}}(P\,\|\,Q(x)) \coloneqq 
\sum\limits_{i=1}^m p(\xi^{(i)})
\big(\log p(\xi^{(i)})-\log q(\xi^{(i)};x)\big).
\] 
In a learning process, the goal is to determine an optimal parameter
$x\in K$ such that $H(P, Q(x))$ or $D_{\operatorname{KL}}(P\,\|\,Q(x))$
is minimized.

\subsection*{Contributions}
In this paper, we propose a hierarchy of moment relaxations 
for log-polynomial optimization.
A generic polynomial optimization problem can be solved globally 
by the Moment-SOS hierarchy \cite{lasserre2001}.
This motivates us to extend this methodology 
to log-polynomial optimization.
We construct a moment relaxation hierarchy for \reff{eq:problem},
where each relaxation problem has a log-linear objective and
semidefinite constraints.
We study the properties of these relaxations and 
provide sufficient conditions to certify their finite convergence 
(i.e., the relaxation is tight at a finite relaxation order).
When the finite convergence is confirmed, 
we give convenient methods to extract the global optimizer(s).
In particular, we analyze the special class of log-polynomial optimization 
given by SOS-concave polynomials.

This paper is organized as follows.
Section~\ref{sec:Pre} introduces the notation and some preliminaries.
Section~\ref{sec:rel} presents the proposed moment relaxation 
approach based on the truncated $K$-moment problems.
In Section~\ref{sec:tight}, we provide theoretical guarantees for 
the relaxation's tightness and discuss conditions under 
which the global optimizers can be extracted.
In Section~\ref{sec:LME}, we introduce Lagrange multiplier expressions (LMEs)
to obtain strengthened moment relaxations.
Numerical experiments and applications are presented in Section~\ref{sec:exp}.
The conclusion and some discussions are given in Section \ref{sec:con}.

\section{Preliminaries}\label{sec:Pre}

\subsection*{Notation}
The symbol $\mathbb{N}$ (resp., $\mathbb{R}$) represents the set of 
nonnegative integers (resp., real numbers).
For a positive integer \(n\), \([n]\coloneqq \{1,\ldots, n\}\) and 
$\mathbb{R}^n$ denotes the $n$-dimensional real Euclidean space.
For a vector $u \in \mathbb{R}^n$, $\|u\|$ denotes the standard 
Euclidean norm and $\delta_{u}$ denotes the Dirac measure supported at $u$.
The $\mathbf{1}_{n}\in\re^n$ denotes the vector of all ones, and
$\mathbf{0}_{n_1 \times n_2}$ denotes the zero matrix of dimension $n_1 \times n_2$.
For notations $\mathbf{1}_{n}$ and $\mathbf{0}_{n_1\times n_2}$,
the subscripts may be omitted if the dimension is clear in the context.
Let $A$ be a real symmetric matrix. 
The inequality $A\succeq 0$ (resp., $A\succ 0$)
means that $A$ is positive semidefinite (resp., positive definite).
Let $x \coloneqq (x_1, \ldots, x_n)$ be a vector of variables.
We use $\mathbb{R}[x]$ to denote the ring of polynomials 
with real coefficients in $x$,
and $\mathbb{R}[x]_d$ its subset of polynomials
with degrees no greater than $d$.
For a function $f(x)$, $\nabla f(x)$ denotes its full gradient
and $\nabla_{x_i} f(x)$ denotes its gradient with respect to $x_i$.
For $\alpha \coloneqq (\alpha_1, \ldots, \alpha_n)\in \mathbb{N}^n$,
denote $x^\alpha \coloneqq x_1^{\alpha_1} \cdots x_n^{\alpha_n}$,
 whose degree is counted by $|\alpha| \coloneqq \alpha_1 + \cdots + \alpha_n$.
Given a degree \(d\), we denote the power set
\[
\mathbb{N}_d^n \coloneqq \{\alpha \in \mathbb{N}^n : |\alpha| \leq d \} .
\]
The column vector of all monomials in $x$ with degrees up to $d$ is denoted as
\begin{equation}\label{eq:xvec}
    [x]_d \coloneqq \begin{bmatrix}
        1 & x_1 & \ldots & x_n & x_1^2 & x_1x_2 & \ldots & x_n^d
    \end{bmatrix}^T.
\end{equation}

\subsection{Nonnegative polynomials}

For polynomials $p, q \in \mathbb{R}[x]$ and 
subsets $I, J \subseteq \mathbb{R}[x]$, 
their product and sum are respectively defined as
\begin{equation*}
p \cdot I \coloneqq \{pq : q \in I\} 
\quad \text{and} \quad 
I+J \coloneqq \{a+b : a \in I, b \in J\}.
\end{equation*}
A subset $I \subseteq \mathbb{R}[x]$ is an \textit{ideal} 
of $\mathbb{R}[x]$ if $I+I \subseteq I$ and $\mathbb{R}[x] \cdot I \subseteq I$.
A polynomial $\sigma \in \mathbb{R}[x]$ is a \textit{sum of squares (SOS)} 
if $\sigma = p_1^2 + \cdots + p_k^2$ for some polynomials $p_i \in \mathbb{R}[x]$. 
The set of all SOS polynomials is denoted by $\Sigma[x]$.

Let  $h = (h_1, \ldots, h_m)$ and $g=(g_1,\ldots, g_t)$
be two polynomial tuples. 
The real zero set of $h$ is the set 
$Z(h) \coloneqq \{x \in \mathbb{R}^n: h_1(x) = \cdots = h_m(x) = 0\}$.
The ideal generated by $h$ is
\begin{equation*}
\mbox{Ideal}[h]\, \coloneqq \, 
h_1 \cdot \mathbb{R}[x] + \cdots + h_m \cdot \mathbb{R}[x] .
\end{equation*}
It is clear that every $p\in \mbox{Ideal}(h)$ vanishes on $Z(h)$.
Consider the semialgebraic set determined by $g$:
\begin{equation*}
S(g) \coloneqq \{x \in \mathbb{R}^n : g_1(x) \geq 0, \ldots, g_t(x) \geq 0\}.
\end{equation*}
A polynomial $f$ is non-negative on $S(g)$ if it can be decomposed as
\begin{equation}\label{eq:fsos}
f = \sigma_0 + \sigma_1 g_1 + \cdots + \sigma_t g_t,
\end{equation}
for some $\sigma_i\in \Sigma[x]$. 
The \textit{quadratic module} generated by $g$ consists of all polynomials
in form of \eqref{eq:fsos}, which reads
\begin{equation*}
\mbox{QM}[g] \coloneqq \Sigma[x] + g_1 \cdot \Sigma[x]
 + \cdots + g_t \cdot \Sigma[x].
\end{equation*}
If \(p\in \operatorname{Ideal}[h] + \operatorname{QM}[g]\),
then $p(x)\ge 0$ for all $x\in Z(h) \cap S(g)$.
But the converse is not always true.
The set $\operatorname{Ideal}[h] + \operatorname{QM}[g]$ 
is said to be \textit{Archimedean} if it contains a polynomial $q$ 
such that $\{x \in \mathbb{R}^n : q(x) \ge 0\}$ is a compact set.
Under this condition, if a polynomial $f$ is strictly 
positive on $Z(h)\cap S(g)$, 
then it must belong to $\operatorname{Ideal}[h]+\operatorname{QM}[g]$.
This result is called \textit{Putinar's Positivstellensatz} \cite{Putinar}.

\subsection{Moment and localizing matrices}
\label{ssc:momloc}

The space of real vectors indexed by $\alpha \in \mathbb{N}_{2k}^n$ 
is denoted by $\mathbb{R}^{\mathbb{N}_{2k}^n}$. 
A vector $y = (y_\alpha)_{\alpha \in \mathbb{N}_{2k}^n}$ 
in this space is called a \textit{truncated multi-sequence} (tms) of degree $2k$. 
For an integer $k \in [0, d]$, 
the $k$-th order \textit{moment matrix} generated by $y$ is defined as
\begin{equation*}
    M_k[y] \coloneqq [y_{\alpha + \beta}]_{\alpha, \beta \in \mathbb{N}_{k}^n}.
\end{equation*}
The rows and columns of $M_k[y]$ are indexed by multi-indices 
from $\mathbb{N}_{k}^n$, typically arranged in graded lexicographic order.
For example, when $n = 2$ and $k = 2$, the matrix $M_2[y]$ is given by
\begin{equation*}
M_2[y] = \begin{bmatrix}
    y_{00} & y_{10} & y_{01} & y_{20} & y_{11} & y_{02} \\
    y_{10} & y_{20} & y_{11} & y_{30} & y_{21} & y_{12} \\
    y_{01} & y_{11} & y_{02} & y_{21} & y_{12} & y_{03} \\
    y_{20} & y_{30} & y_{21} & y_{40} & y_{31} & y_{22} \\
    y_{11} & y_{21} & y_{12} & y_{31} & y_{22} & y_{13} \\
    y_{02} & y_{12} & y_{03} & y_{22} & y_{13} & y_{04}
\end{bmatrix}.
\end{equation*}
Every polynomial $f\in\re[x]_{2k}$ can be identified with its coefficient vector
$\operatorname{vec}(f) \coloneqq (f_{\alpha})_{\alpha\in \mathbb{N}_{2k}^n}$,
i.e., $f(x) =  \operatorname{vec}(f)^T[x]_{2k}$. 
Thus, a tms $y \in \mathbb{R}^{\mathbb{N}_{2k}^n}$ defines a 
linear functional on $\mathbb{R}[x]_{2k}$ as
\begin{equation}\label{eq:bilinear}
	\langle f, y \rangle \coloneqq 
	\sum_{\alpha \in \mathbb{N}_{2k}^n} f_\alpha y_\alpha = \operatorname{vec}(f)^T y.
\end{equation}
For a polynomial $q \in \mathbb{R}[x]_{2k}$, 
let $s \coloneqq k - \lceil \deg(q)/2 \rceil$.
The product $q(x) [x]_s [x]_s^T$ is a $\binom{n+s}{s}$-by-$\binom{n+s}{s}$
symmetric polynomial matrix. 
Its expansion can be written as
\begin{equation*}
    q(x) [x]_s [x]_s^T = \sum_{\alpha \in \mathbb{N}_{2k}^n} x^\alpha  Q_\alpha
\end{equation*}
for some constant symmetric matrices $Q_\alpha$. 
The $k$-th order \textit{localizing matrix} associated with 
the polynomial $q$ and the tms $y \in \mathbb{R}^{\mathbb{N}_{2k}^n}$ is defined as
\begin{equation*}
    L^{(k)}_q [y] \coloneqq \sum_{\alpha \in \mathbb{N}_{2k}^n} y_\alpha Q_\alpha.
\end{equation*}
For the special case that $q=1$ (the constant unit polynomial), 
this definition recovers the moment matrix, i.e., $L^{(k)}_1 [y] = M_k[y]$.

\subsection{Truncated $K$-moment problems}

A tms $y = (y_\alpha)_{\alpha \in \mathbb{N}^n_{2k}}$ is said to 
\textit{admit} a Borel measure $\mu$ on $K$ if
\begin{equation} \label{eq:y_admits_meas}
 y_\alpha = \int_K x^\alpha d \mu \quad \text{for all} \quad
 \alpha \in \mathbb{N}^n_{2k}.
\end{equation}
Such a measure $\mu$ is called a \textit{representing measure} of $y$.
The support of $\mu$, denoted by $\operatorname{supp}(\mu)$, 
is the smallest closed set $T \subseteq \mathbb{R}^n$ 
such that $\mu(\mathbb{R}^n \setminus T) = 0$.
A measure is called \textit{finitely atomic} if its support is a finite set.
Specifically, if $\operatorname{supp}(\mu) = \{v_1, \ldots, v_r\}$, 
then the measure $\mu$ is called $r$-atomic and can be written as
\[
    \mu = \sum_{i=1}^r \lambda_i \delta_{v_i} \quad 
    \text{for some positive weights } \lambda_1, \dots, \lambda_r > 0.
\]
For a tms $y \in \mathbb{R}^{\mathbb{N}_{2k}^n}$ to admit a representing measure
$\mu$ on $K$ as in \eqref{eq:y_admits_meas}, 
it is necessary that the following matrix conditions hold (see \cite{MPO}):
\begin{equation} \label{eq:loc_mom}
    M_k[y] \succeq 0, \quad L_{c_j}^{(k)}[y] \succeq 0 \,(j \in \mathcal{I}), \quad
    L_{c_j}^{(k)}[y] = 0 \, (j \in \mathcal{E}).
\end{equation}

\section{A semidefinite relaxation hierarchy}
\label{sec:rel}

In this section, we propose an approach of moment relaxations
to solve the log-polynomial optimization problem \reff{eq:problem}.

\subsection{A convex relaxation}
To begin with, we introduce a natural convex relaxation of the 
log-polynomial optimization \eqref{eq:problem}.
Define the degree
\begin{equation}\label{eq:d0d}
d  \coloneqq  \max\,\{ \lceil \deg(p)/2\rceil, 
	\, \lceil\deg(c_{\mc{E}})/2\rceil,
	\, \lceil\deg(c_{\mc{I}})/2\rceil
 \}.
\end{equation}
Each polynomial in \eqref{eq:problem} belongs to the polynomial ring $\re[x]_{2k}$ for all $k\ge d$.
By replacing each involving monomial $x^{\alpha}$ with the auxiliary moment variable $y_{\alpha}$, 
we obtain the following equivalent moment reformulation of \eqref{eq:problem}:
\begin{equation}\label{eq:momequiv}
\left\{\begin{array}{rl}
f_{\max}\coloneqq \max\limits_{z} & \sum\limits_{i=1}^m a_i \log \langle p_i, z\rangle\\
\st & z = [x]_{2d}\mbox{ for some $x\in K$}.
\end{array}
\right.
\end{equation}
Here $\langle \cdot , \cdot \rangle$ is the bilinear operation 
defined in \eqref{eq:bilinear}, 
$[x]_{2d}$ is the $2d$-th degree monomial vector as in \eqref{eq:xvec},
and $K$ is the feasible set given in \eqref{eq:K}.
The problem \eqref{eq:momequiv} has a concave objective function 
since each $a_i>0$ and the logarithm is concave.
However, its feasible set is typically nonconvex and may not have a nonempty interior. 
Thus, most general nonlinear optimization methods are not applicable 
to solve \eqref{eq:momequiv}. 
To address this issue, we define the conic hull generated 
by the feasible set of \eqref{eq:momequiv}:
\begin{equation}  \label{Rd(K)}
\mathscr{R}_d(K) \coloneqq \Big\{ z = \sum\limits_{j=1}^l \lambda_j [v_j]_{2d} :
\lambda_j \ge  0, \,\, v_j \in K,\,\, l \in \mathbb{N} \Big\}.
\end{equation}
It follows that $\mathscr{R}_d(K)\cap \{z_0=1\}$ is the convex hull 
of the set $\{[x]_{2d}: x\in K\}$.
Given that $p_i(x)\ge 0$ for all $i\in [m]$ and $x\in K$, 
each $\log \langle p_i,z\rangle$ is nonnegative on $\mathscr{R}_d(K)\cap \{z_0=1\}$. 
This is because any $z\in \mathscr{R}_d(K)\cap \{z_0=1\}$ 
can be decomposed as $z = \sum_{j=1}^l \lambda_j[v_j]_{2d}$ 
for some $l\in \N$ with all $v_j\in K$ and 
$\lambda_j\ge 0,\,\lambda_1+\cdots+\lambda_l=1$. Hence,
\[
    \langle p_i, z\rangle 
    = \Big\langle p_i,\sum\limits_{i=1}^l \lambda_j[v_j]_{2d}\Big\rangle 
    =  \sum\limits_{j=1}^l \lambda_j \langle p_i, [v_j]_{2d}\rangle 
    = \sum\limits_{j=1}^m \lambda_j p_i(v_j) \ge 0.
\]
This leads to a natural convex relaxation of \eqref{eq:momequiv}:
\begin{equation}  \label{eq:relax_meas}
\left\{\begin{array}{rl}
f_{\operatorname{mom}} \coloneqq \max\limits_{z} 
&  \sum\limits_{i=1}^m a_i \log \langle p_i,z\rangle   \\
\st & z_0 = 1,\, z  \in  \mathscr{R}_d(K) .
\end{array}\right.
\end{equation}
It is clear that $f_{\max}\le f_{\operatorname{mom}}$.
In particular, when $f_{\max} = f_{\operatorname{mom}}$, 
the optimization problem \eqref{eq:relax_meas} is said to 
be a {\it tight} relaxation of \eqref{eq:momequiv}.
We remark that the relaxation gap $f_{\operatorname{mom}}-f_{\max}$ can be
arbitrarily large.
This phenomenon is illustrated in the following example.
\begin{example}\label{ex:gap}
    Consider the univariate log-polynomial optimization problem
    \begin{equation}\label{eq:ex_gap}
        \left\{\begin{array}{cl}
        \max\limits_{x\in\re} & 0.5\log(1+\epsilon-x)+0.5\log(1+\epsilon+x)\\
        \st & 1-x^2 = 0,
        \end{array}
        \right.
    \end{equation}
    where $\epsilon>0$ is a small scalar. 
    Since the feasible set $K = \{1,-1\}$ only contains two points, 
    we can express the conic hull $\mathscr{R}_1(K)$ explicitly, 
    thus obtain the convex relaxation of \eqref{eq:ex_gap}:
    \[
    \left\{\begin{array}{cl}
    	\max\limits_{z\in\re^3} & 0.5\log(1+\epsilon-z_1)+0.5\log(1+\epsilon+z_1)\\
    	\st &  z_0=z_2 = 1,\, -1\le z_1\le 1.
    \end{array}
    \right.
    \]
    The optimal values of \eqref{eq:ex_gap} and its convex relaxations are
    respectively derived as 
    \[
    f_{\max} = 0.5\log(\epsilon)+0.5\log(2+\epsilon),
    \quad
    f_{\operatorname{mom}} = \log(1+\epsilon).
    \]
    It is clear that the relaxation gap 
    $f_{\operatorname{mom}} - f_{\max}\ge 0-0.5\log(\epsilon)\to \infty$
    as $\epsilon\to 0$.
\end{example}
We now provide sufficient conditions for the relaxation 
\eqref{eq:relax_meas} to be tight. 
\begin{prop}\label{prop:gap_maxmom}
    Suppose $K$ is nonempty. 
    Then the convex relaxation \eqref{eq:relax_meas} is a tight relaxation of 
    the log-polynomial optimization \eqref{eq:problem} if one of the following conditions holds.
    \begin{enumerate}
        \item[(i)] The optimal value of the moment reformulation 
        \eqref{eq:momequiv} equals that of
        \begin{equation}\label{eq:min_gamma}
            \left\{\begin{array}{cl}
            \min & \gamma_1+\cdots+\gamma_m\\
            \st & \gamma_i- a_i\log p_i(x)\ge 0 \mbox{ for all $x\in K$}\, (i\in[m]),\\
            & \, \gamma_1,\ldots, \gamma_m\in\re.
            \end{array}
            \right.
        \end{equation}
        
        \item[(ii)] There exists a maximizer 
        $z^* = \sum_{j=1}^l\lambda_j[v_j^*]_{2d}$ of the convex relaxation
        \eqref{eq:relax_meas} with all $\lambda_j\ge 0$ 
        and $v_j^*\in K$ such that $p(v_1^*) = p(v_2^*) = \cdots = p(v_l^*)$.
    \end{enumerate}  
\end{prop}
\begin{proof}
	The original log-polynomial optimization \eqref{eq:problem} is
	equivalent to its moment reformulation \eqref{eq:momequiv}.
	Let $f_{\max},\, f_{\operatorname{mom}}$ denote the optimal values 
	of \eqref{eq:momequiv} and \eqref{eq:relax_meas} respectively.
	It suffices to show $f_{\max}\ge f_{\operatorname{mom}}$.

	(i) If $f_{\max} = \infty$, 
	then $f_{\operatorname{mom}}\ge f_{\max} = \infty$. 
	Suppose $f_{\max}<\infty$ is the optimal value of \eqref{eq:min_gamma}. 
	Let $\hat{\gamma} = (\hat{\gamma}_1,\ldots, \hat{\gamma}_m)$ 
	be the minimizer of \eqref{eq:min_gamma}.
	Then $f_{\max} = \mathbf{1}^T\hat{\gamma}$.
	Note that the inequality $\gamma_i\ge a_i\log p_i(x)$ 
	is equivalent to $e^{\gamma_i/a_i}\ge p_i(x)$. 
	The feasibility of $\hat{\gamma}$ in \eqref{eq:min_gamma} ensures 
	that for all $i\in[m]$,
	\begin{align*}
		e^{\hat{\gamma}_i/a_i}\ge \max\limits_{x\in K}\, p_i(x) 
 		= \max\limits_{x\in K}\,\langle p_i,[x]_{2d}\rangle.
	\end{align*}
	Since $\mathscr{R}_d(K)\cap \{z_0=1\}$ is the 
	convex hull of $\{[x]_{2d}: x\in K\}$,
	the above implies
	\[
		e^{\hat{\gamma}_i/a_i}-\langle p_i, z\rangle \ge 0
		\,\Leftrightarrow\, 
		\hat{\gamma}_i-a_i\log \langle p_i, z\rangle \ge 0
	\]
	for all feasible point of \eqref{eq:relax_meas}.
	Thus, $f_{\max} = \mathbf{1}^T\hat{\gamma}\ge f_{\operatorname{mom}}$.

	(ii) Under given conditions, if $p_i(v_1^*) = \cdots = p_i(v_l^*)$ 
	for all $i\in [m]$, then 
    \[\begin{aligned}
        f_{\operatorname{mom}}  
        & = \sum\limits_{i=1}^m a_i\log \big\langle p_i, 
        	\sum\limits_{j=1}^l \lambda_j [v_j^*]_{2d}\big\rangle
        = \sum\limits_{i=1}^m a_i\log \big(\sum\limits_{j=1}^l 
        	\lambda_j\langle p_i,  [v_j^*]_{2d}\rangle\big)\\
        &= \sum\limits_{i=1}^m a_i\log \big(\sum\limits_{j=1}^l 
        	\lambda_j p_i(v_j^*)\big)
        = \sum\limits_{i=1}^m a_i\log \langle p_i, [v_1^*]_{2d}\rangle\le f_{\max}.
    \end{aligned}\]
	Therefore, the convex relaxation \eqref{eq:relax_meas} is tight 
	if condition (i) or (ii) holds.
\end{proof}

For the special case that $m=1$, 
the log-polynomial optimization \eqref{eq:problem} is equivalent to 
the polynomial optimization of minimizing $p_1(x)$ over $K$. 
The following result is directly implied from Proposition~\ref{prop:gap_maxmom}.
\begin{cor}
    The relaxation \eqref{eq:relax_meas} is tight if $m=1$.
\end{cor}

In practice, the problem \eqref{eq:relax_meas} is still difficult to solve.
This is because the moment cone $\mathscr{R}_d(K)$ typically 
does not have a convenient parametric expression for multivariable variables.
On the other hand, the moment cone $\mathscr{R}_d(K)$ 
is often approximated by semidefinite constraints.
This motivates us to further construct 
a hierarchy of semidefinite relaxations of \eqref{eq:relax_meas}.

\subsection{A hierarchy of moment relaxations}

Recall that $K$ is determined by polynomial constraints
\[
 	c_j(x) = 0\,(j\in\mc{E}),\quad c_j(x)\ge 0\,(j\in\mc{I}).
\]
Let $z\in\mathscr{R}_d(K)$ be an arbitrary truncated moment vector. 
For all $k\ge d$, there exists a $k$-th order tms extension of $z$
such that
\[
y=(y_{\alpha})\in\re^{\N_{2k}^n}\,\,\mbox{with}\,\,y_{\alpha} = z_{\alpha}
\, \forall \alpha\in\N_n^d\quad \mbox{and}
\]
\[
M_{k}[y]\succeq 0,\quad L_{c_j}^{(k)}[y] = 0\,(j\in \mc{E}),\quad
L_{c_j}^{(k)}[y]\succeq 0\, (j\in \mc{I}).
\]
In the above, $M_k[y]$ is the $k$-th order moment matrix and each
$L_{c_j}^{(k)}[y]$ is the $k$-th order 
localizing matrix associated with $c_j$, 
which are introduced in Subsection~\ref{ssc:momloc}.
For increasing values of $k = d,d+1,\ldots$, 
we can build a hierarchy of moment relaxations of \eqref{eq:relax_meas} 
with semidefinite constraints: 
\begin{equation}    \label{eq:relax_log}
\left\{\begin{array}{rl}
   f_{\operatorname{mom},k} \coloneqq \displaystyle \max\limits_{y} 
   & \sum\limits_{i=1}^m a_i \log \langle p_i, y \rangle \\
   \st  &  y_0 = 1,\,y \in \mathbb{R}^{\mathbb{N}_{2k}^n},\, M_k [y] \succeq 0,\\
   &   L_{c_j}^{(k)}[y] = 0\,(j \in \mathcal{E}),\\ 
    & L_{c_j}^{(k)}[y] \succeq 0 \,(j \in \mathcal{I} ),
\end{array}\right.
\end{equation}
where the integer $k$ is called the {\it relaxation order}.
To distinguish with the previous convex relaxation \eqref{eq:relax_meas},
we call the problem \eqref{eq:relax_log} the 
$k$-th order {\it moment relaxation}
and the corresponding hierarchy the {\it moment hierarchy}.
Note that each moment relaxation \eqref{eq:relax_log} is convex, 
but distinct from standard semidefinite program 
as it has a log-linear objective function.
In computational practice, it can be solved globally by conic problem 
solvers with interior point methods.
Let $f_{\operatorname{mom},k}$ denote the optimal value of \eqref{eq:relax_log} at 
the $k$-th relaxation order. 
By previous analysis, it is clear that 
\[
    f_{\operatorname{mom}}\le \cdots\le f_{\operatorname{mom},k+1}
    \le f_{\operatorname{mom},k}\le \cdots \le f_{\operatorname{mom},d}.
\]
In the following, we show this relaxation hierarchy exhibits 
asymptotic convergence.
\begin{theorem}   \label{thm:asymp}
	Assume that $\mbox{Ideal}[c_{\mc{E}}]+\qmod{c_{\mc{I}}}$ 
	is Archimedean, and the optimal value of \eqref{eq:relax_meas} 
	equals that of \eqref{eq:min_gamma}. 
	Then $f_{\operatorname{mom},k}\to f_{\operatorname{mom}}$ as $k\to\infty$. 
\end{theorem}
\begin{proof}
	Let $\hat{\gamma} = (\hat{\gamma}_1,\ldots, \hat{\gamma}_m)$ 
	denote the optimizer of the minimization problem \eqref{eq:min_gamma}.
	Then for any $\epsilon>0$, $\hat{\gamma}_i+\epsilon - a_i\log p_i(x)>0$ 
	for all $x\in K$ and $i\in [m]$, equivalently,
	\[
		e^{(\hat{\gamma}_i+\epsilon)/a_i}-p_i(x)>0 
		\mbox{ on $K$ for all $i\in[m]$}.
	\]
	Assume $\mbox{Ideal}[c_{\mc{E}}]+\qmod{c_{\mc{I}}}$ is Archimedean.
	By Putinar's Positivstellensatz, there exists $N_0(\epsilon)>0$ such that 
	for all $i\in[m]$ and $k\ge N_0(\epsilon)$,
	\[
		e^{(\hat{\gamma}_i+\epsilon)/a_i}-p_i(x)\in 
		\mbox{Ideal}[c_{\mc{E}}]_{2k}+\qmod{c_{\mc{I}}}_{2k},
	\]
	where $\mbox{Ideal}[c_{\mc{E}}]_{2k}+\qmod{c_{\mc{I}}}_{2k}$ is 
	the $2k$-th degree truncation of 
	$\mbox{Ideal}[c_{\mc{E}}]+\qmod{c_{\mc{I}}}$.
	By \cite[Theorem~2.5.2]{MPO}, every feasible point of \eqref{eq:relax_log} 
	belongs to the dual cone of $\mbox{Ideal}[c_{\mc{E}}]_{2k}+\qmod{c_{\mc{I}}}_{2k}$.
	This implies
	\[
		\langle e^{(\hat{\gamma}_i+\epsilon)/a_i}\cdot 1-p_i, y\rangle 
		= e^{(\hat{\gamma}_i+\epsilon)/a_i}-\langle p_i, y\rangle \ge0 
	\]
	for all feasible point of \eqref{eq:relax_log}.
	Since $\langle p_i,y\rangle \ge 0$ 
	(for the special case that $\langle p_i,y\rangle =0$, 
	the following inequality still holds with $\log(0)=-\infty$),
	by taking the logarithm,
	\[
		\hat{\gamma}_i+\epsilon-a_i\log\langle p_i, y\rangle \ge 0
		\quad \forall i\in[m].
	\]
	Summing up the above inequality over $i\in[m]$, we obtain
	\[
		m\epsilon+\mathbf{1}^T\hat{\gamma}_i
		-\sum\limits_{i=1}^m a_i\log\langle p_i, y\rangle\ge 0
	\]
	for all feasible point of \eqref{eq:relax_log}.
	Note that $f_{\operatorname{mom}} = \mathbf{1}^T\hat{\gamma}$ since 
	\eqref{eq:relax_meas} and \eqref{eq:min_gamma} share the same optimal value.
	Thus, $f_{\operatorname{mom}}\le f_{\operatorname{mom},k}
	\le f_{\operatorname{mom}} + m\epsilon$ for all $k\ge N_0(\epsilon)$.
	Letting $\epsilon\to 0$, we have $N_0(\epsilon)\to \infty$, 
	thus $f_{\operatorname{mom},k}\to f_{\operatorname{mom}}$ as $k\to \infty$.
\end{proof}
Recall that $f_{\max}$ represents the optimal value of the original log-polynomial optimization.
Combining Proposition~\ref{prop:gap_maxmom} and Theorem~\ref{thm:asymp} together, 
we can get the following result.
\begin{cor}\label{cor:asymp}
    Assume that $\mbox{Ideal}[c_{\mc{E}}]+\qmod{c_{\mc{I}}}$ is Archimedean 
    and the condition (i) of Proposition~\ref{prop:gap_maxmom} holds. 
    Then $f_{\operatorname{mom},k}\to f_{\max}$ as $k\to \infty$.
\end{cor}

\section{Tightness analysis of moment relaxations}\label{sec:tight}

The moment relaxation \eqref{eq:relax_log} is said to be 
{\it tight} (or to have finite convergence) to \eqref{eq:problem} 
if they have the same optimal value at a finite relaxation order.
In this section, we study how to certify such tightness and 
how to obtain the maximizers when the tightness is confirmed.  
In particular, we study the special class of log-polynomial optimization 
given by SOS-concave polynomials, of which the finite convergence
is always achieved.

\subsection{Certifying tightness of moment relaxations}

Consider $k$ is the relaxation order.
Let $f_{\operatorname{mom},k}$ denote the optimal value 
of the moment relaxation \eqref{eq:relax_log},
$f_{\operatorname{mom}}$ the optimal value 
of the convex relaxation \eqref{eq:relax_meas},
and $f_{\max}$ the optimal value of the 
original log-polynomial optimization \eqref{eq:problem}. 
Since $f_{\max}\le f_{\operatorname{mom}}\le f_{\operatorname{mom,k}}$, 
a necessary condition for $f_{\max} = f_{\operatorname{mom,k}}$ 
is $f_{\max} = f_{\operatorname{mom}}$.
This holds if an optimal solution $y^*$ of \eqref{eq:relax_log} 
satisfies $y^*|_{2d}\in \mathscr{R}_d(K)$, 
where $y^*|_{2d}$ denotes the $2d$-th degree truncation of $y^*$.
To verify such a membership, we can apply a convenient rank condition 
called {\it flat truncation}, which reads
\begin{equation}\label{eq:flat_truncation}
	\exists t\in[d,k]\quad \mbox{s.t.}\quad    
	\rank\,M_t[y^*] = \rank\,M_{t-d}[y^*].
\end{equation}
Under the flat truncation condition, $y^*$ admits a unique $K$-representing 
atomic measure supported by $r=\rank\, M_t[y^*]$ distinct points. 
In other words, there exist positive scalars $\lmd_1\ddd \lmd_r\in\re$ 
and distinct points $v_1^* \ddd v_r^* \in K$ such that
\begin{equation}\label{eq:ft_decomp}
	y^*|_{2t} = \lmd_1 [v_1^*]_{2t} + \cdots + \lmd_r [v_r^*]_{2t},
\end{equation}
where $y^*|_{2t}$ is the $2t$-th order truncation of $y^*$.
In particular, such a decomposition \eqref{eq:ft_decomp} is unique by 
\cite[Theorem~2.7.7]{MPO}, up to reordering. 
Since $t\ge d$, the decomposition \eqref{eq:ft_decomp} ensures that 
$y^*|_{2d}\in \mathscr{R}_d(K)$.
We then summarize sufficient conditions for \eqref{eq:relax_log} 
to be a tight relaxation for \eqref{eq:problem}.

\begin{theorem}\label{tm:parallel_all}
	Suppose $y^*$ is an optimizer of the $k$-th order moment relaxation \eqref{eq:relax_log}.
	If $y^*$ satisfies the flat truncation condition 
	\eqref{eq:flat_truncation} for some $t\in[d,k]$,
	then $f_{\operatorname{mom},k} = f_{\operatorname{mom}}$,
	and $y^*|_{2t}$ admits the decomposition in \eqref{eq:ft_decomp} 
	for $r=\rank\, M_t[y^*]$ distinct points $v_1^*,\ldots, v_r^*\in K$.
	If in addition that $p(v_1^*) =  \cdots = p(v_r^*)$, 
	then $f_{\max} = f_{\operatorname{mom}} = f_{\operatorname{mom},k}$
	and all $v_1^* \ddd v_r^*$ are global maximizers of \reff{eq:problem}.
\end{theorem}
\begin{proof}
	Under the flat truncation condition, 
	$z^* = y^*|_{2d}$ is a feasible point of \eqref{eq:relax_meas}.
	Since each $\deg(p_i) \le 2d$, we have
	\begin{equation}
    	f_{\operatorname{mom,k}} = \sum_{i=1}^m a_i \log \langle p_i,y^*\rangle 
    		= \sum_{i=1}^m a_i \log \langle p_i,z^*\rangle 
    		\leq f_{\operatorname{mom}}.
	\end{equation}
	Combined the above inequality with 
	$f_{\operatorname{mom}} \le f_{\operatorname{mom,k}}$, 
	we conclude that $f_{\operatorname{mom}} = f_{\operatorname{mom,k}}$.
	The decomposition \eqref{eq:ft_decomp} is implied by the 
	flat truncation condition \cite[Theorem~2.7.7]{MPO}.
	If $p(v_1^*) = \cdots = p(v_r^*)$,
	then $f_{\max} = f_{\operatorname{mom}}$ by Proposition~\ref{prop:gap_maxmom}.
	Consequently, $f_{\max} = f_{\operatorname{mom}} = f_{\operatorname{mom,k}}$.
	To show that $v_l^*$ for each $l\in[r]$ is a global maximizer, 
	it suffices to show $f(v_l^*) = f_{\operatorname{mom,k}} = f_{\max}$.
	Indeed, we have
	\begin{equation*}
    	f_{\operatorname{mom,k}} 
    	= \sum_{i=1}^m a_i \log \Big( \sum_{j=1}^r \lambda_j p_i(v_j^*) \Big) 
   	 	= \sum_{i=1}^m a_i \log \big( p_i(v_l^*) \big) = f(v_l^*)
	\end{equation*}
	for every $l\in[r]$. It follows that $v_1^*, \ldots, v_r^*$ 
	are all global maximizers of \eqref{eq:problem}.
\end{proof}

Notice that $\rank\, M_d[y^*] = 1$ if and only if $y^*|_{2d} = [v^*]_{2d}$ 
with $v^* = (y^*_{e_1}\ddd y^*_{e_n})$. 
For this special case, all conditions in Theorem~\ref{tm:parallel_all} 
are satisfied automatically.

\begin{cor}\label{cor:rank1}
	Suppose $y^*$ is an optimizer of the $k$-th order moment relaxation\eqref{eq:relax_log}.
	If $\rank\, M_d[y^*] = 1$, then $f_{\max} = f_{\operatorname{mom,k}}$
	and $v^* = (y^*_{e_1}\ddd y^*_{e_n})$ is a global maximizer of \eqref{eq:problem}.
\end{cor}

We remark that sufficiency of the flat truncation condition for log-polynomial optimization
is quite different from that for standard polynomial optimization.
Suppose $y^*$ is an optimizer of the $k$-th order moment relaxation \eqref{eq:relax_log}.
If it satisfies the flat truncation condition \eqref{eq:flat_truncation}, 
then \eqref{eq:relax_log} is a tight relaxation of \eqref{eq:relax_meas}, 
which implies that $f_{\operatorname{mom}} = f_{\operatorname{mom,k}}$.
However, the convex relaxation \eqref{eq:relax_meas} 
may not be a tight relaxation of the moment reformulation \eqref{eq:momequiv},
as shown in Example~\ref{ex:gap}.
Consequently, even if $y^*$ admits the decomposition \eqref{eq:ft_decomp} 
with candidate solutions $v_1^*,\ldots, v_r^*\in K$, 
it remains possible that $f_{\max} < f_{\operatorname{mom}} = f_{\operatorname{mom,k}}$.
Interestingly, in such scenarios, the points $v_i^*$ may or may not be global maximizer(s) 
of \eqref{eq:problem}. 
We illustrate both possibilities in the following example.

\begin{example}\label{ex:counter}
Consider the log-polynomial optimization problem:
\begin{equation}\label{eq:counter}
\left\{\begin{array}{cl}
 \max\limits_{x\in\re} & f(x)\coloneqq \log((x-1)^2+a) + \log((x+1)^2+a)\\
 \st  & x(x-1)(x+1) = 0,
\end{array}\right.
\end{equation}
where $a>0$ is a parameter. The feasible set $K = \{1, -1, 0\}$
and the degree bound $d=2$.
It is easy to evaluate
\[ f(1) = \log (a(4+a)), \quad f(-1) = \log (a(4+a)),\quad  f(0) = 2\log (1+a). \]
Consider the moment relaxation \eqref{eq:relax_log} with 
the relaxation order $k=3$.
The equality constraint in \eqref{eq:counter} is equivalent to
$x^3-x=0$, which induces linear equality constraints 
$y_5 = y_3 = y_1$ and $y_6 = y_4 = y_2$ in the moment relaxation. 
By eliminating variables, 
the $3$rd-order moment relaxation of \eqref{eq:counter} can be simplified to 
\begin{equation}\label{eq:relax_counter}
\left\{\begin{array}{cl}
    \displaystyle \max_{(y_1,y_2)} &
    \log(y_2-2y_1+1+a ) +  \log(y_2+2y_1+1+a )   \\
   \st &
   y_1,y_2\in\re,\,\left[ \begin{array}{cccc}
   1 & y_1 & y_2 & y_1 \\
   y_1 & y_2 & y_1 & y_2\\
   y_2 & y_1 & y_2 & y_1\\
   y_1 & y_2 & y_1 & y_2\\
   \end{array}
   \right] \succeq 0.
\end{array}\right.
\end{equation}
In the above, the objective function equals $\log((y_2^2+1+a)^2-4y_1^2)$ 
and the positive semidefinite constraint implies that 
$0\le y_1^2 \le y_2^2 \le y_2 \le 1.$
Thus, for all feasible point of \eqref{eq:relax_counter},
it is satisfied that
\[
\log((y_2^2+1+a)^2-4y_1^2)\le \log((1+1+a)^2-4(0)) = 2\log(a+2).
\]
In particular, the above equality holds at $y_1^* = 1$ and $y_2^* = 0$.
Hence the optimal value and optimizer of the 
$3$rd-order moment relaxation of \eqref{eq:counter} are
\[
f_{\operatorname{mom,3}} = 2\log(a+2),\quad
y^* = (1,0,1,0,1,0,1),
\] 
where $y^* = 0.5[u_1]_6 + 0.5[u_2]_6$ is supported 
on feasible points $u_1 = 1$, $u_2 = -1$.

(i) When $a = 1$, we have
\[
	f(1) = f(-1) = \log(5)>f(0) = \log(4).
\]
In this case, $f_{\max} = \log(5)$ and 
both $u_1 = 1$ and $u_2 = -1$ are maximizers of (\ref{eq:counter}).
However, $f_{\operatorname{mom,3}} = \log(9)$, 
which is strictly bigger than $f_{\max}$.
Indeed, one can check that 
$f_{\operatorname{mom}} = f_{\operatorname{mom,k}} = \log(9)$ 
for all $k\ge 3$, indicating a nonzero gap between the moment relaxation hierarchy 
and the original log-polynomial optimization problem.

(ii) When $a = 0.1$, we have
\[
	f(1) =  f(-1) = \log (0.41)< f(0) = \log (1.21).
\]
This implies that $f_{\max} = \log(1.21)$,
which is achieved at the unique optimizer $u_3 = 0$.
In this case, $f_{\operatorname{mom,3}} = \log (4.41)>f_{\max}$ 
and $u_1,u_2$ are no longer global maximizers of \eqref{eq:counter}.
It can be further verified that the relaxation \eqref{eq:relax_log} 
remains non-tight for all $k\ge 3$.
\end{example}

Proposition~\ref{prop:gap_maxmom} gives two sufficient conditions for 
$f_{\max} = f_{\operatorname{mom}}$,
both of which require numerical verification.
Recall that moment relaxations for convex polynomial optimization
always exhibit finite convergence.
This motivates us to study log-polynomial optimization defined with
concave/convex polynomials.

\subsection{Concave log-polynomial optimization}

The problem \eqref{eq:problem} is called a {\it concave} log-polynomial 
optimization problem if $K$ is a convex set and 
each $p_i$ is concave on $K$, i.e., 
for all $u,v\in K$, $i\in[m]$ and $\lambda\in[0,1]$,
\[
	\lambda u+(1-\lambda)v\in K,\quad
	p_i(\lambda u+(1-\lambda)v)\ge  \lambda p_i(u)+(1-\lambda)p_i(v).
\]
Concave log-polynomial optimization has tight moment relaxations 
under flat truncation conditions.

\begin{theorem}   \label{thm:concave}
Assume \eqref{eq:problem} is concave and 
$y^*$ is an optimizer of its $k$-th order moment relaxation \reff{eq:relax_log}.
If $y^*$ satisfies the flat truncation condition \reff{eq:flat_truncation},
then $f_{\max} = f_{\operatorname{mom}} = f_{\operatorname{mom,k}}$ 
and $v^* = (y^*_{e_1}\ddd y^*_{e_n})$
is a global maximizer of \reff{eq:problem}.
\end{theorem}
\begin{proof}
Suppose $y^*$ satisfies the flat truncation condition \eqref{eq:flat_truncation}
for some $t\in [d,k]$. Let $r = \rank\,M_r[y^*]$. 
There exist positive scalars $\lambda_i$ and distinct points $v_i^*\in K$
for $i\in[r]$ such that 
\[
	y^*|_{2d} = \lambda_1[v_1^*]_{2d}+\cdots +\lambda_r[v_r^*]_{2d}.
\]
Since $y_0^* = 1$ and \eqref{eq:problem} is a 
concave log-polynomial optimization problem, 
we have $\lambda_1+\cdots+\lambda_r = 1$ and 
$v^* = (y_{e_1}^*, \ldots, y_{e_n}^*) 
= \lambda_1v_1^*+\cdots+\lambda_rv_r^*\in K.$
This implies
\[
	\langle p_i,y^*\rangle 
	= \sum_{j=1}^r \lambda_j \langle p_i, [v_j^*]_{2d}\rangle 
	= \sum_{j=1}^r \lambda_j p_i(v_j^*) \le p_i(v^*)
\]
for all $i\in[m]$.
By assumption, $y^*$ is an optimizer of \eqref{eq:relax_log}.
Since the logarithm function is monotonically increasing, 
we further have
\[
    f_{\operatorname{mom,k}}  
    = \sum\limits_{i=1}^m a_i \log \langle p_i, y^* \rangle  
    \le \sum\limits_{i=1}^m a_i \log p_i(v^*)
    \le f_{\max}.
\]
This implies that $f_{\max} = f_{\operatorname{mom,k}}$ 
and $v^*$ is a maximizer of \reff{eq:problem}.
\end{proof}

We remark that Theorem~\ref{tm:parallel_all} and 
Theorem~\ref{thm:concave} describe two different classes of 
log-polynomial optimization problems that yield tight moment relaxation, 
while both require flat truncation conditions.
Specifically, Theorem~\ref{tm:parallel_all} provides conditions 
under which all points $v_1^*, \ldots, v_r^*$ 
in the support of the atomic measure admitted by $y^*|_{2d}$ 
are maximizers of \eqref{eq:problem}.
In contrast, Theorem~\ref{thm:concave} focuses on concave log-polynomial optimization, 
where these $v_1^*, \ldots, v_r^*$ are not necessarily 
global maximizers of \eqref{eq:problem}.

However, determining the convexity or concavity of polynomials 
with degree four or higher is NP-hard \cite{AAA13}. 
In computational practice, SOS-convex/concave polynomials, 
which form a subset of convex/concave polynomials, 
are of great interest since they can be verified by
solving a semidefinite program \cite[Lemma~7.1.3]{MPO}.
A polynomial $q(x)$ is said to be {\it SOS-convex} 
if there exists a matrix polynomial $R(x)$ such that its Hessian matrix $\nabla^2 q = R^TR$,
and it is said to be {\it SOS-concave} if $-q$ is SOS-convex.
We then show that $f_{\max} = f_{\operatorname{mom,k}}$ for every $k\ge d$
if \eqref{eq:problem} is defined by SOS-concave polynomials.

\begin{theorem}\label{thm:sosconcave}
Assume $c_j$ are SOS-concave for all $j\in\mc{I}$, 
$c_{\mc{E}}$ is a linear polynomial tuple and 
each $p_i(x)$ for $i\in[m]$ are SOS-concave for all $x\in K$.
Then $f_{\operatorname{mom,k}} = f_{\max}$ for all $k\ge d$.
In addition, if \reff{eq:relax_log} has an optimizer $y^*$,
then $v^* = (y^*_{e_1}\ddd y^*_{e_n})$ is a global maximizer of \eqref{eq:problem}.
\end{theorem}
\begin{proof}
Given $k\ge d$, let $y$ be a feasible point of \eqref{eq:relax_log}
and denote $\pi(y)\coloneqq (y_{e_1}\ddd y_{e_n})$.
For every $j\in \mc{I}$, $\langle c_j,y\rangle \ge 0$ 
since $L_{c_j}^{(k)}[y]\succeq 0$.
For every $j\in \mc{E}$, since $c_j$ is linear, 
we have $c_j(\pi(y)) = 0$ as it corresponds to the 
$(1,1)$-th entry of $L_{c_j}^{(k)}[y]$.
Since $M_d[y]\succeq 0$ and $y_0=1$, 
the SOS-concavity and Jensen's inequality \cite[Theorem~7.1.6]{MPO} imply
\[ 
	\langle p_i,y\rangle \le p_i(\pi(y)),\quad 0
	\le \langle c_j,y\rangle \le c_j(\pi(y)), 
\]
for all $i\in[m]$ and $j\in \mc{I}$. 
Thus, we have $\pi(y)\in K$ and
\begin{equation}\label{eq:jensen}
    \sum\limits_{i=1}^m a_i \log \langle p_i,y\rangle  
    \le  \sum\limits_{i=1}^m a_i \log p_i(\pi(y)) 
    \leq  f_{\max}.
\end{equation}
Note that the relaxation order $k$ and the 
feasible point $y$ can be chosen arbitrarily.
By maximizing \eqref{eq:jensen} with respect to $y$ 
over the feasible set of \eqref{eq:relax_log}, 
we obtain $f_{\operatorname{mom,k}} \le f_{\max}$, 
which implies $f_{\operatorname{mom,k}} = f_{\max}$ for all $k\ge d$.
Suppose \reff{eq:relax_log} has an optimizer $y^*$.
Then by Theorem~\ref{thm:concave}, 
$v^* = \pi(y^*)$ is a global maximizer of \reff{eq:problem}.
\end{proof}

Theorem~\ref{thm:sosconcave} offers an effective perspective for 
optimizing a product of SOS-concave functions over a convex set.
Consider a polynomial optimization problem of the form
\begin{equation}\label{eq:prod_pop}
\left\{\begin{array}{cll}
 \max\limits_{x\in\re^n} & \theta(x)\coloneqq p_1(x)p_2(x)\cdots p_{m}(x)\\
 \st & c_j(x) = 0 \,\, (j \in \mathcal{E}),\\
     & c_j(x) \geq 0 \,\, (j \in \mathcal{I}).
\end{array}\right.
\end{equation}
Let $K$ denote the feasible set of \eqref{eq:prod_pop}.
Assume $p_1\ddd p_m$ (not necessarily distinct) are SOS-concave over $K$,
each $c_{j}$ for $j\in\mc{E}$ is linear, and every $c_{j}$ for $j\in\mc{I}$ is SOS-concave.
When $m > 1$, its objective function $\theta(x)$ is typically not concave.
Consequently, the standard Moment-SOS relaxations of 
\eqref{eq:prod_pop} may not be tight at their lowest relaxation order.
If each $p_i(x)$ for $i\in[m]$ is positive over $K$, then \eqref{eq:prod_pop} can 
be equivalently solved as the log-polynomial optimization problem
\[
\max\limits_{x\in K}\, \log \theta(x) = \sum\limits_{i=1}^m \log p_i(x).
\]
By Theorem~\ref{thm:sosconcave}, this reformulated problem has 
a tight moment relaxation at its lowest relaxation order.

\section{Tighter relaxations with LMEs}
\label{sec:LME}

In this section, we introduce a technique called Lagrange multiplier expressions
to construct a kind of tighter moment relaxations of \eqref{eq:problem}
compared to \eqref{eq:relax_log}.

For convenience, suppose that $\mc{E}\cup \mc{I} = [\ell] \coloneq \{1,\ldots, \ell\}$.
Let $f(x)$ be the objective function and $x^*$ the maximizer of \eqref{eq:problem}.
Under some certain constraint qualification conditions, 
there exists a vector of Lagrange multipliers 
$\lambda = (\lambda_j)$ for $j = 1\ddd \ell$, 
such that $(x^*,\lmd)\in\mc{K}$, where
\begin{equation}    \label{eq:KKT}
 \mathcal{K} \coloneqq 
 \left\{ (x, \lambda)\in K\times \re^{\ell}
\left|\begin{array}{c}
\nabla f(x) + \sum\limits_{j \in \mathcal{E} \cup \mathcal{I}} \lambda_j \nabla c_j(x) = 0,\\
\lambda_j\ge 0,\, \lambda_j c_j(x) \ge 0\, (j \in \mathcal{I})
\end{array}\right.\right\}.
\end{equation}
The constraints in \eqref{eq:KKT} are called the 
Karush-Kuch-Tucker (KKT) optimality conditions.
Each $(x, \lambda) \in \mathcal{K}$ is called a critical pair, 
where $x$ alone is called a critical point.
Under some constraint qualifications, 
all maximizers of the log-polynomial optimization \eqref{eq:problem} are critical points.
Suppose there exists a polynomial or rational tuple 
$\tau = (\tau_1,\ldots, \tau_{\ell})$ such that 
\begin{equation}\label{eq:lme}
\lambda_j = \tau_j(x)\quad \mbox{for all}\quad (x,\lambda)\in \mc{K}.
\end{equation}
Then we can get a parametric approximation of critical points by
replacing $\lambda$ in \eqref{eq:KKT} with $\tau$.
Such a $\tau$ is called a Lagrange multiplier expression (LME).
LMEs serve as an efficient tool for constructing tight relaxations 
in polynomial optimization \cite{Nie2019}.
Recently, they have been applied to challenging optimization problems 
given by polynomials, including bilevel optimization, 
generalized Nash equilibrium problems and variational inequalities
\cite{Choi2024,NieSunTangZ,nie2023convex,NieTangZgnep21,NieWangZBilevel21}.

\begin{example}\label{ex:lmes}
We present explicit LMEs for box, simplex and ball constraints.

 \begin{enumerate}

\item[(i)] If $K=\{x\in\re^n: a_j \leq x_j \leq b_j,\, j\in[n]\}$, 
then $\ell = 2n$ and \eqref{eq:lme} is satisfied for 
\(\tau = (\tau_1,\ldots,\tau_{2n})\) with
\begin{equation*}
    \tau_{2j-1} = \frac{\nabla_{x_j} f (x)( x_j - b_j)}{b_j - a_j},\,\, 
    \tau_{2j} = \frac{\nabla_{x_j} f (x)(x_j - a_j )}{b_j - a_j}\quad 
    \forall j\in[n].
\end{equation*}

\item[(ii)] If $K=\{x\in\re^n: 1 - \mathbf{1}^T x \ge 0,x \geq 0\}$,
then $\ell = n+1$ and \eqref{eq:lme} is satisfied for 
\[
	\tau = x^T\nabla f(x)\mathbf{1} - (0,\, \nabla_{x_1} f(x),
		\, \ldots,\, \nabla_{x_n} f(x)),
\]
where $\mathbf{1}\in\re^{n}$ is the vector of all ones.

\item[(iii)] If $K = \{x\in\re^n: \|x-a\| \leq b\}$, 
then $\ell=1$ and \eqref{eq:lme} is satisfied for
\begin{equation*}
    \tau = \frac{(x - a)^T \nabla f(x)}{2b^2}.
\end{equation*}

\end{enumerate}
\end{example}

Next, we introduce how to apply LMEs to construct tighter
moment relaxations for \eqref{eq:problem}.
Consider the log-polynomial optimization problem (\ref{eq:problem})
with box, simplex or ball constraints.
Then, polynomial LMEs $\tau_j(x)$ for $j\in\mc{E}\cup \mc{U}$ are 
explicitly presented in Example~\ref{ex:lmes}. Recall that 
\[
f(x) = \sum\limits_{i=1}^m a_i \log p_i(x),\quad
\nabla f(x) = \sum\limits_{i=1}^m \frac{a_i}{p_i(x)}\nabla p_i(x).
\]
For convenience, denote $p^{\mathbf{1}} = p_1p_2\cdots p_m$ and
\[
f_i(x) \coloneqq \frac{p^{\mathbf{1}}(x)}{p_i(x)},\quad \forall i\in[m].
\]
Assume $p>0$ over $K$. 
Then a feasible point $x\in K$ is a critical point of \eqref{eq:problem} 
if and only if it satisfies 
\[
	\left\{\begin{array}{l}
	\sum\limits_{i=1}^m a_i f_i(x)\nabla p_i(x) 
		+ p^{\mathbf{1}}(x)\sum\limits_{j\in \mc{E}\cup \mc{I}} 
		\tau_j(x)\nabla c_j(x) = 0,\\
	\tau_j(x)\ge 0\,(j\in\mc{I}),\quad  \tau_j(x)c_j(x)=0\, (j\in\mc{I}).
	\end{array}\right.
\]
This polynomial system provides strengthening constraints for \eqref{eq:problem}.
For convenience, denote the polynomial tuples
\[	\begin{array}{l}
	\Phi\coloneq \{ c_{\mc{E}} \} \cup \Big\{ \sum\limits_{i=1}^m 
		a_i f_i(x)\nabla p_i(x) + p^{\mathbf{1}}(x)\sum\limits_{j\in \mc{E}\cup \mc{I}} 
		\tau_j(x)\nabla c_j(x) \Big\} \cup 
	\{\tau_jc_j: j\in \mc{I}\},  \\
	\Psi\coloneq \{ c_{\mc{I}} \} \cup \{ \tau_j: j\in \mc{I} \}.
	\end{array}
\]
For a vector $\phi$ of polynomials, 
$\{\phi\}$ denotes the polynomial set of entries of $\phi$.
We can equivalently reformulate \eqref{eq:problem} as
\begin{equation}\label{eq:problemhatLME}
\left\{\begin{array}{cl}
    \max\limits_{x\in\re^n} 
    & \sum\limits_{i=1}^m a_i \log p_i(x)\\
    \st
    & \varphi(x) = 0 \, (\varphi\in\Phi),\\
    & \psi(x)\ge 0 \, (\psi\in\Psi).
\end{array}\right.
\end{equation}
It is worth noting that the feasible set of \eqref{eq:problemhatLME},
denoted as
\[
	K_1 \coloneqq \{x\in\re^n: \varphi(x)=0\,(\varphi\in \Phi),\, 
	\psi(x)\ge 0\,(\psi\in \Psi)\},
\]
is usually a proper subset of $K$. 
The problem \eqref{eq:problemhatLME} is also 
a log-polynomial optimization problem.
Thus, one gets its convex relaxation similar to (\ref{eq:relax_meas}):
\begin{equation}  \label{eq:relax_meas_lme}
	\left\{\begin{array}{rl}
	f_{\operatorname{lme}} \coloneqq \max\limits_{z} 
	&  \sum\limits_{i=1}^m a_i \log \langle p_i,z\rangle   \\
	\st & z_0 = 1,\, z  \in  \mathscr{R}_d(K_1) .
	\end{array}\right.
\end{equation}
Let $f_{\operatorname{lme}}$ denote the optimal value of \eqref{eq:relax_lme}.
Since $K_1\subseteq K$, we must have 
$f_{\operatorname{lme}} \le f_{\operatorname{mom}}$,
where $f_{\operatorname{mom}}$ is the optimal value of the 
standard convex relaxation \eqref{eq:relax_meas}.
Moreover, its moment relaxation hierarchy can be built similarly 
to \eqref{eq:relax_log}. Let 
\[
	d_1 \coloneqq \max\{ \lceil \deg(q)\rceil: q\in \Phi\cup \Psi\}.
\]
The $k$-th order (for each $k\ge d_1$) moment relaxation 
of \eqref{eq:relax_meas_lme} is
\begin{equation}\label{eq:relax_lme}
	\left\{\begin{array}{rl}
   f_{\operatorname{lme},k} \coloneqq \max\limits_y 
   & \sum\limits_{i=1}^m a_i \log \langle p_i, y \rangle \\
   \st 
   & y_0 = 1,\, y \in \mathbb{R}^{\mathbb{N}_{2k}^n},\, 
   		M_k [y] \succeq 0,\\
   & L_{\varphi}^{(k)}[y] = 0\,(\varphi\in\Phi),\\ 
   & L_{\psi}^{(k)}[y] \succeq 0 \,(\psi  \in \Psi).
	\end{array}\right.
\end{equation}
Let $f_{\operatorname{lme},k}$ denote the optimal value of \eqref{eq:relax_lme}.
Since all constraints of \eqref{eq:relax_log} are included in \eqref{eq:relax_lme}, 
the feasible set of \eqref{eq:relax_log} is a subset of that for \eqref{eq:relax_lme}.
Thus, one has $f_{\operatorname{lme},k}\le f_{\operatorname{mom},k}$ 
for all $k\ge d_1$.

The tighter relaxation can be similarly constructed
for more general cases.
Suppose the linear independence constraint qualification condition (LICQC) 
holds for all $x\in K$ and all $p_i$ are positive on $K$.
Then there exist rational LMEs with denominators strictly positive on $K$, 
by Proposition~3.6 in \cite{nie2023convex}. 
In summary, we have the following result.
\begin{prop}\label{prop:rel}
Suppose there exist rational LMEs $\tau = (\tau_1,\ldots, \tau_l)$ that
satisfies \eqref{eq:lme} with denominators strictly positive on $K$. 
Then convex problems \eqref{eq:relax_meas_lme} and \eqref{eq:relax_lme} are relaxations of 
\eqref{eq:problem}, whose optimal values satisfy
\[
\begin{array}{ccccc}
f_{\operatorname{mom},k} & \ge & f_{\operatorname{mom}} & \ge & f_{\operatorname{max}} \\
\rotatebox[origin=c]{-90}{$\ge$} & & \rotatebox[origin=c]{-90}{$\ge$} & &  \\
f_{\operatorname{lme},k} & \ge & f_{\operatorname{lme}} & \ge & f_{\operatorname{max}}.
\end{array}
\]
\end{prop}
We remark that it is possible that 
$f_{\operatorname{lme}} < f_{\operatorname{mom}}$, 
and thus, $f_{\operatorname{mom},k}> f_{\max}$ for all $k$, 
but the moment relaxation hierarchy (\ref{eq:relax_lme}) exploiting LMEs is tight.
We refer to Example~\ref{ex:44} for such an exposition.
In addition, we would like to remark that analogous conclusions in Theorems~\ref{thm:asymp},~\ref{tm:parallel_all},~\ref{thm:concave},~\ref{thm:sosconcave} 
and Corollaries~\ref{cor:asymp},~\ref{cor:rank1} 
also hold for the tighter relaxation (\ref{eq:relax_lme}). 
We omit the detailed proofs for brevity and present an example 
to better illustrate the performance difference between \eqref{eq:relax_log} 
and \eqref{eq:relax_lme} in the following exposition.

\begin{example}\label{ex:ABO}
Consider the log-polynomial optimization problem derived from the 
ABO blood group system \cite{EMBook}:
\begin{equation}\label{eq:ABO}
\left\{\begin{array}{cl}
    \max\limits_{x\in\re^3} 
    & 182 \log (x_1^2+2x_1x_3)+ 60 \log (x_2^2+2x_2x_3)\\
    &\qquad + 17 \log(2x_1x_2) + 176 \log (x_3^2)\\
    \st & 1-\mathbf{1}^Tx \ge 0,\,\, x \geq \mathbf{0},
\end{array}\right.
\end{equation}
where $\mathbf{1}\in\re^3$ is the vector of all ones and 
$\mathbf{0}\in\re^3$ is the vector of all zeros.
This optimization problem has simplex constraints, 
and its LMEs are explicitly given in Example~\ref{ex:lmes} (ii). 
We compare the numerical performance of the standard 
moment relaxation \eqref{eq:relax_log} against the tighter 
relaxation \eqref{eq:relax_lme}. 
All computations were performed using \texttt{Yalmip} \cite{Yalmip} 
with the SDP solver {\tt MOSEK} \cite{mosek}.
For the tighter moment relaxation \eqref{eq:relax_lme}, 
the flat truncation condition \eqref{eq:flat_truncation} 
holds for $t=k=3$ with $r=1$.
By Corollary~\ref{cor:rank1}, we obtain the optimal value 
and optimizer of \eqref{eq:ABO} as
\[
f_{\max} = f_{\operatorname{lme},3} = -492.5353,\quad
x^* = (0.2644, 0.0932, 0.6424).
\]
The detailed computational results are presented in Table~\ref{tab:ex:ABO}. 
\begin{table}[htb]
    \centering
    \begin{tabular}{|c|c|c|c|c|c|}
    \hline
       $k$  &   2 & 3 & 4 & 5 & 6\\
       \hline
       $f_{\operatorname{mom},k}$  &  -491.8158 &  -491.8158 &  -491.8158 &  -491.8158 &  -491.8158\\
       \hline
       $f_{\operatorname{lme},k}$ &  -492.2927 & -492.5353 & -492.5353 & -492.5353 & -492.5353\\
       \hline
    \end{tabular}
    \caption{Comparison between $f_{\operatorname{mom},k}$ and $f_{\operatorname{lme},k}$ in Example~\ref{ex:ABO}}
    \label{tab:ex:ABO}
\end{table}
It is clear that the tighter relaxation \eqref{eq:relax_lme}
is more efficient for solving the log-polynomial optimization problem \eqref{eq:ABO}
compared to the standard moment relaxation \eqref{eq:relax_log}.
\end{example}

In practice, choosing between \eqref{eq:relax_log} and \eqref{eq:relax_lme} 
for solving log-polynomial optimization can be tricky. 
Note that the degrees $d_1\ge d$, and $d_1$ depends on $m$.
Typically, we prefer the tighter relaxation \eqref{eq:relax_lme} 
when $m$ is small and \eqref{eq:problem} has simple constraints such as 
box, simplex, and ball constraints. 
This is because in such cases, the lowest relaxation orders $d,d_1$ 
are close and \eqref{eq:problem} has convenient LMEs.
Consequently, solving \eqref{eq:relax_lme} at a small relaxation order 
is expected to provide superior upper bounds compared to \eqref{eq:relax_log}.
However, as $m$ increases, the gap $d_1-d$ can grow substantially 
due to the existence of the products $p^{\bf 1}$ and $f_i$. 
In this case, solving \eqref{eq:relax_lme} can be computationally expensive, 
even at its initial relaxation order $k = d_1$. 
Therefore, for larger $m$, it is often more efficient to compute 
using the standard moment hierarchy \eqref{eq:relax_log}.

\section{Numerical experiments and applications}
\label{sec:exp}

In this section, we present explicit numerical examples 
and applications of log-polynomial optimization.
The moment relaxations \reff{eq:relax_log} are solved by the software
\texttt{Yalmip} \cite{Yalmip} with the SDP solver \texttt{MOSEK} \cite{mosek}.
The computations were implemented in \MATLAB R2024b on a laptop equipped with
a 12th Gen Intel(R) Core(TM) i7-1270P 2.20GHz CPU and 32GB RAM.
To enhance readability, the computational results are displayed 
with four decimal digits.

\subsection{Explicit numerical examples}

\begin{example}\label{ex:88}
Consider the log-polynomial optimization problem:
\begin{equation*}
\left\{\begin{array}{cl}
    \max\limits_{x \in \mathbb{R}^5} & 
    \frac{30}{218} \log p_1(x) + \frac{97}{218} \log p_2(x) + \frac{91}{218} \log p_3(x) \\
    \st & \| x \| \leq 1,
\end{array}\right.
\end{equation*}
where
\begin{equation*}
\begin{array}{rl}
   p_1(x)  & = (x_3+x_4)^2 + \left(1+2(x_1+x_5)+3(x_2+x_4) \right)^2 + 0.01, \\
   p_2(x)  & = (x_2+x_4-x_5)^2 + (2x_2+3x_5)^2 + 0.02,  \\
   p_3(x)  & = (x_1+x_3+x_4)^2 + (x_1-x_3+x_4)^2 + 0.03.
\end{array}
\end{equation*} 
It is clear that the objective function is well defined 
over the feasible set. 
It took 1.10 seconds to solve the relaxation \reff{eq:relax_log} at $k=2$. 
The corresponding moment optimizer satisfies the flat truncation 
condition \reff{eq:flat_truncation} at $t=2$ with $r=1$.
By Corollary \ref{cor:rank1}, the moment relaxation is tight.
We obtain the global optimal value
$f_{\max} = f_{\operatorname{mom,2}} = 1.6216$, 
and the global maximizer
\begin{equation*}
    x^* = (0.4717, 0.4518, 0.0036, 0.5276, 0.5432).
\end{equation*} 
\end{example}

\begin{example}\label{ex:44}
Consider the log-polynomial optimization problem:
\begin{equation}\label{eq:44}
\left\{\begin{array}{cl}
    \max\limits_{x \in \mathbb{R}^3} & \sum\limits_{i=1}^{6} a_i \log  p_i(x)  \\
    \st & 1 - \mathbf{1}^T x \geq 0,\,\, x \geq \mathbf{0},
\end{array}\right.
\end{equation}
where $a = (a_1,\ldots, a_6)$ is a given coefficient vector and
\begin{equation*}
	\begin{array}{lll}
	p_1(x) = x_1^3 + 3x_1^2x_2 + 3x_1^2 x_3, 
		& p_2(x) = 3x_1x_2^2 + 6x_1x_2x_3 , 
		& p_3(x) = 3x_1x_3^2,\\
	p_4(x) = x_2^3 + 3x_2^2x_3, 
		& p_5(x) = 3x_2x_3^2 , 
		& p_6(x) = x_3^3 ,
	\end{array}
\end{equation*}
The problem \eqref{eq:44} has simplex constraints, 
and its LMEs are explicitly presented in Example~\ref{ex:lmes} (ii).

(i) Consider the coefficient vector
\[
a=
( 0.0968,\,    0.1419,\,    0.2194,\,    0.0839,\,    0.2839 ,\,   0.1742).
\]

For the standard moment relaxation \reff{eq:relax_log}, 
the lowest relaxation order is $d=2$.
When the relaxation order $k=4$, it took around 1.77 seconds to solve 
$f_{\operatorname{mom},4} =  -1.7181$, of which
the associated moment optimizer satisfies the flat truncation condition with 
$t=3$ and $r=2$, admitting the decomposition \reff{eq:ft_decomp} with
\[	\begin{array}{ll}
	v_1^* = (0.2195,0.1929,0.5877),  & f(v_1^*) = -1.7308,\\
	v_2^* = (0.0537,0.3100,0.6363), & f(v_2^*) = -2.0655.
	\end{array}
\]
By Theorem~\ref{tm:parallel_all}, 
we have $f_{\operatorname{mom}} = f_{\operatorname{mom},4}$.
But this does not guarantee the optimality of the original problem.
Note that $f(v_1^*), f(v_2^*)$ provide lower bounds for 
the true optimal value $f_{\max}$. 
Thus, we derive
\[
 |f_{\max}-f_{\operatorname{mom}}| 
 \le |f(v_1^*)-f_{\operatorname{mom},4}|\leq 0.0127. 
\]

For the tighter moment relaxation \reff{eq:relax_lme},
the lowest relaxation oder is $d_1 = 3$.
At this lowest relaxation order  $k = 3$,
it took around 1.24 seconds to get $f_{\operatorname{lme},3} = -1.7194$, 
of which the associated moment optimizer satisfies 
the flat truncation condition with $t=3$ and $r=1$.
By Corollary \ref{cor:rank1}, the moment relaxation is tight.
We get the global optimal value and optimizer 
\begin{equation*}
	f_{\max} = f_{\operatorname{lme},3} = -1.7194,\quad
    x^* = (0.1873,0.2161,0.5967).
\end{equation*}

(ii) 
Let the coefficient vector $a$ be randomly generated by the 
\MATLAB function $\texttt{rand}$.
Table~\ref{tab:ex:44} reports the computational results for five instances. 
We use $k_{\operatorname{mom}}$ and $k_{\operatorname{lme}}$ 
to denote the relaxation orders of \eqref{eq:relax_log} 
and \eqref{eq:relax_lme}, respectively, at which flat truncation holds.
For relaxation \eqref{eq:relax_log}, 
the value $r$ is the rank of the moment matrix associated with the 
moment optimizer that satisfies the flat truncation condition.
The column `gap' lists the smallest value difference 
between $f_{\operatorname{mom}}$ and $f(v_i^*)$, 
where $v_i^*$ is any point in the support of the atomic measure admitted
by the moment optimizer of \eqref{eq:relax_log}.
For the tighter moment relaxation \eqref{eq:relax_lme}, 
its moment optimizer satisfies flat truncation with $t=3$ and $r=1$ for all five instances.
By Corollary \ref{cor:rank1}, the moment relaxation is tight, 
i.e., $f_{\max} = f_{\operatorname{lme}}$.
\begin{table}[htb]
    \centering
    \begin{tabular}{|c|c|c|c|c|c|c|c|}
       \hline
       \multirow{2}{*}{Instance} & \multicolumn{4}{c}{Relaxation~\eqref{eq:relax_log}} & \multicolumn{2}{|c|}{Relaxation~\eqref{eq:relax_lme}}\\
       \cline{2-7} 
       & $k_{\operatorname{mom}}$ & $r$& $f_{\operatorname{mom}}$ & gap & $k_{\operatorname{lme}}$ & $f_{\max} =f_{\operatorname{lme}}$ \\
       \hline
       \#1 & 4 & 2 & -4.6833 & 0.2033 & 3 & -4.6968\\
       \hline
       \#2 & 4 & 3 & -3.9524 & 0.0225 & 3 &  -3.9658\\
       \hline
       \#3 & 5 & 3 &   -2.7490 & 0.5442 & 3 & -2.8979\\
       \hline
       \#4 & 6 & 3 & -6.6434 & 0.3871 & 3 &  -6.7018\\
       \hline
       \#5 & 5 & 4 & -7.0067 & 0.5815 & 3 &   -7.0672\\
       \hline
    \end{tabular}
    \caption{Computational results for Example~\ref{ex:44} (ii)}
    \label{tab:ex:44}
\end{table}

\end{example}

\begin{example}\label{ex:92}
Consider the polynomial optimization problem:
\[
\left\{\begin{array}{cl}
\max\limits_{x\in\re^5} 
& \theta(x) = \big((x_3+x_4+x_5)^2-x_1x_2\big)^{20}\big(8-(x_5-x_2)^2\big)^{25}\\
\st & \|x\|\le 2,\, (x_3+x_4+x_5)^2-x_1x_2\ge 0,\\
&  8-(x_5-x_2)^2\ge 0.
\end{array}
\right.
\]
Since $\deg(\theta) = 45$, it is computationally expensive 
to directly apply Moment-SOS relaxations. 
Instead, we can solve the problem by reformulating it into 
an equivalent log-polynomial optimization problem.
There are two possible reformulations:
\begin{equation}\label{eq:ex92}
\left\{\begin{array}{cl}
    \max\limits_{x \in \mathbb{R}^5} &  20\log\left((x_3+x_4+x_5)^2-x_1x_2\right)+25\log\left(8-(x_5-x_2)^2\right)\\
    \st & \|x\|\le 2,\, (x_3+x_4+x_5)^2-x_1x_2\ge 0,\\
    &  8-(x_5-x_2)^2\ge 0;
\end{array}\right.
\end{equation}
\begin{equation}
    \label{eq:ex92_2}
\left\{\begin{array}{cl}
 \max\limits_{x \in \mathbb{R}^5} &  20\log\left((x_3+x_4+x_5)^2-x_1x_2\right)\\
 & \qquad +25\log\left(2\sqrt{2}-(x_5-x_2)\right)
 +25\log\left(2\sqrt{2}+(x_5-x_2)\right)\\
    \st & \|x\|\le 2,\, (x_3+x_4+x_5)^2-x_1x_2\ge 0,\\
    &  8-(x_5-x_2)^2\ge 0.
\end{array}\right.
\end{equation}
We consider the standard moment relaxation \eqref{eq:relax_log} of 
\eqref{eq:ex92}--\eqref{eq:ex92_2},
since these log-polynomial optimization problems have
relatively complex constraints.

(i) For the problem \eqref{eq:ex92}, it took around $1$ second to get 
$f_{\operatorname{mom},2} = 99.7252$ at the relaxation order $k=2$.
The corresponding moment optimizer satisfies the flat truncation 
condition \eqref{eq:flat_truncation} at $t=2$.
It admits the decomposition \eqref{eq:ft_decomp} with 
\[\begin{array}{ll}
    v_1^* = ( 0.0673,
   -0.3654,
   -1.2400,
   -1.2400,
   -0.8870),\\
   v_2^* =  (-0.0673,
    0.3654,
    1.2400,
    1.2400,
    0.8870).
\end{array}
\]
Denote the polynomials $p_1(x) = (x_3+x_4+x_5)^2-x_1x_2$
and $p_2(x) = 8-(x_5-x_2)^2$.
Since $p_1,p_2$ are even functions, 
it is clear that $p_1(v_1^*) = p_1(v_2^*)$ and $p_2(v_1^*) = p_2(v_2^*)$.
Therefore, by Theorem~\ref{tm:parallel_all}, we have
\[
f_{\max} = f_{\operatorname{mom},2} = 99.7252
\]
and both $v_1^*,v_2^*$ are global maximizers of the original optimization problem.

(ii) For the problem \eqref{eq:ex92_2}, 
it took 1.21 second to solve $f_{\operatorname{mom},2} = 101.6842$ 
at the relaxation order $k=2$.
The corresponding moment optimizer satisfies the flat truncation 
condition \eqref{eq:flat_truncation} at $t=2$.
It admits the decomposition \eqref{eq:ft_decomp} with 
\[\begin{array}{ll}
    v_1^* = (  0.0001,
   -0.0002,
   -1.1546,
   -1.1546,
   -1.1546),\\
   v_2^* =  (-0.0001,
    0.0002,
    1.1546,
    1.1546,
    1.1546).
\end{array}
\]
Denote the polynomials $p_1(x) = (x_3+x_4+x_5)^2-x_1x_2$, 
$p_2(x) = 2\sqrt{2}-(x_5-x_2)$ and $p_3(x) = 2\sqrt{2}+(x_5-x_2)$.
Since $p_2(v_1^*)\not=p_2(v_2^*)$, 
the conditions of Theorem~\ref{tm:parallel_all} are not satisfied.
Based on the computational results in (i), the moment relaxation of
\eqref{eq:ex92_2} is not tight.
\end{example}

\begin{example}\label{ex:dim10}
Consider the log-polynomial optimization of the form
\begin{equation*}
\left\{\begin{array}{cl}
    \max\limits_{x \in \mathbb{R}^{10}} & \sum\limits_{i=1}^{5} a_i \log p_i(x) \\
    \st & p_i(x) - i \geq 0, \quad (i \in [5]),\\
    & 6 - (x_6-x_7-x_8)^2 - x_9^2 \geq 0, \\
    & 7 - (x_1-x_{10})^2 - (x_2-x_9)^2 \geq 0,\\
    & (x_1 + x_2 + x_3)^2 + (x_4+x_5)^2 - 8 \geq 0,\\ 
    & (x_2+x_4-x_6+x_8-x_{10})^2 - 9 \geq 0,
\end{array}\right.
\end{equation*}
where each $p_i(x)$ takes the form
\[
p_i(x)\coloneqq \alpha_i - \big(\beta_i(x_{2i-1} + b_i)^2 
+ \gamma_i(x_{2i}+c_i)^2 + \delta_i(x_{2i-1}+d_i)(x_{2i}+e_i)\big).
\]
Let weights $a_i$ and parameters of $p_i$ be given as follows.
\smallskip
\begin{center}
    \begin{tabular}{|c|c|c|c|c|c|c|c|c|c|}
    \hline
        $i$ & $a_i$ & $\alpha_i$ & $\beta_i$ & $\gamma_i$ & $\delta_i$  & $b_i$ & $c_i$ & $d_i$ & $e_i$\\ \hline
        1 & 3 & 10 & 2 & 3 & -3  & 0.5 & 0.5 & -0.5 & 0.5\\ \hline
        2 & 2 & 12 & 3 & 2 &  -1  & -0.6 & 0.1 & -0.6 & -0.1\\ \hline
        3 & 1 & 11 & 4 & 1 & 0  & 0.7 & -0.2 & 0 & 0 \\ \hline
        4 & 1 & 15 & 2 & 5 & -2 & -0.8 & 0.3 & -0.8 & 0.3 \\ \hline
        5 & 5 & 13 & -3 & 2 & 0  & -0.9 & 0.4 & 0 & 0 \\ \hline
    \end{tabular}
\end{center}
\smallskip
By checking the positive semidefiniteness of their Hessian matrices, 
one can easily verify that $p_1, \ldots, p_4$ are concave, whereas $p_5$ is not.
It took 1.95 second to solve the relaxation \reff{eq:relax_log} for $k=2$.
The flat truncation condition (\ref{eq:flat_truncation}) holds for $t=2$ and $r=1$.
By Corollary \ref{cor:rank1}, 
the moment relaxation is tight.
We get the global optimal value $f_{\max} = f_{\operatorname{mom,2}} =  36.3994,$
and the optimizer
\begin{equation*}
        \begin{split}
            x^* = (&-1.4038, -1.5957, 0.4259, -0.4182, -0.7548, \\
                   & 0.6579, 0.9798, -0.3904, -2.4485, -0.0622).
        \end{split}
    \end{equation*}
\end{example}

\begin{example} \label{ex:dim12}
Consider the log-polynomial optimization problem:
\begin{equation*}
\left\{\begin{array}{cl}
    \max\limits_{x \in \mathbb{R}^{12}} & \sum\limits_{i=1}^{10} \log p_i(x) \\
    \st & (x_1 + x_2 + x_3 + x_4 + x_5 + x_6)^2 \leq 15,\\
    & (x_7 + x_8 + x_9 + x_{10} + x_{11} + x_{12})^2 \leq 15,\\
    & (x_1 - x_3 + x_5 - x_7 + x_9 - x_{11})^2 \leq 8, \\
    & (x_2 - x_4 + x_6 - x_8 + x_{10} - x_{12})^2 \leq 8, \\
    & (x_1 + x_{12})^2 + (x_2 + x_{11})^2 \leq 9,\\
    & (x_3 - x_{10})^2 + (x_4 - x_9)^2 + (x_5-x_8)^2 \leq 9 \\
    & \sum\limits_{i=1}^6 x_{2i}^2  + 2(x_1^2 + x_3^2 + x_7^2 + x_9^2) + 3(x_5^2 + x_{11}^2) \leq 20,\\
    & \sum\limits_{i=1}^6 (x_{2i-1} - x_{2i}) \leq 5, \, (x_1 + x_6 + x_{12})^2 \leq 4, \\
    & x_1 - x_{12} \leq 3, \, \mathbf{1}^T x \geq 0,
\end{array}\right.
\end{equation*}
In the above, for each $i = 1\ddd 10$,
\begin{equation*}
    p_i(x) \coloneqq \alpha_i - \beta_i \big(h_i(x)\big)^4,
\end{equation*}
where for $i=1,\ldots, 5$,
\[
    h_i(x) \coloneqq (x_{2i}+a_i) + (x_{2i+1}+b_i),
\]
and for $i=6,\ldots, 10$ ($j_1\bmod j_2$ denotes the remainder of $j_1$ modulo $j_2$)
\[ \begin{array}{ll}
    h_i(x) \coloneqq \;& (-1)^{i-4}(x_{i-5}+a_i) 
    + (-1)^{i+1} (x_i +b_i) \\
    & + (-1)^{(i+4) \bmod 12} (x_{(i+4) \bmod 12 + 1} + c_i).
\end{array}
 \]
Moreover, parameters in $p_i,h_i$ are given in the following table:
\smallskip
\begin{center}
    \begin{tabular}{|c|c|c|c|c|c|c|c|c|c|c|}
    \hline
    $i$ & 1 & 2 & 3 & 4 & 5 & 6 & 7 & 8 & 9 & 10 \\ \hline
    $\alpha_i$ & 20 & 25 & 18 & 22 & 30 & 15 & 28 & 19 & 21 & 26 \\ \hline
    $\beta_i$ & 1.5 & 1.2 & 2.5 & 1.1 & 1.6 & 1.9 & 2.3 & 1.7 & 2.4 & 1.4 \\ \hline
    $a_i$ & -0.2 & -1 & 0.4 & 0.4 & -1.1 & -0.8 & -0.3 & -1.7 & 1.2 & 1.7 \\ \hline
    $b_i$ & -1.6 & -0.4 & -1 & 0.8 & -1.5 & -0.7 & 0 & -1 & -1.9 & 0.9 \\ \hline
    $c_i$ & - & - & - & - & - & -1.3 & 0.7 & -0.7 & 0.8 & -0.5 \\ \hline
\end{tabular}
\end{center}
\smallskip
By checking the positive semidefiniteness of Hessian matrices, 
one can easily verify that all $p_i$ are SOS-concave 
and the feasible set is convex.
So this log-polynomial optimization problem is concave.
It took $3.19$ seconds to solve the standard moment 
relaxation \reff{eq:relax_log} with the relaxation order $k=2$.
The corresponding moment optimizer does not satisfy the flat truncation (\ref{eq:flat_truncation}). 
However, by Theorem \ref{thm:sosconcave}, 
the moment relaxation \eqref{eq:relax_log} is always tight 
for concave log-polynomial optimization problems. Thus,
$
f_{\max} = f_{\operatorname{mom,2}} =  30.3426,
$
and the obtained global maximizer is 
\begin{equation*}
        \begin{split}
            x^* = (& 1.5452, 0.2145, -0.8337,  1.7120,-1.1664, 0.4258, \\
                   & -0.2089, 0.3101, 0.5313, -2.1910, 0.4054, -0.7443).
        \end{split}
    \end{equation*}
\end{example}

\subsection{Applications}
\label{sec:app}

\begin{example}[Paternity analysis problem \cite{McCulloch87}] 
\label{ex:offspring}
Consider a single locus with $n$ alleles involving a single mother 
and $n$ possible fathers.
Suppose there are $t$ possible genotypes.
Let $x_j$ denote the probability that the $j$th male is the biological father 
of a randomly chosen offspring for $j=1\ddd k$.
Our goal is to estimate the probabilities $x=(x_1,\ldots, x_n)$ 
using genotype data from a potentially multiply-parented litter.
This can be done by solving 
\begin{equation}\label{eq:offspring}
\left\{\begin{array}{cl}
\max\limits_{x\in\re^n} & \sum\limits_{i=1}^t N_i \log \Big( \sum\limits_{j=1}^n P_{ij}x_j \Big)\\
\st & \mathbf{1}^Tx = 1,\,\, x \geq \mathbf{0},
\end{array}\right.
\end{equation}
where each $N_i$ denotes the number of offspring in the litter with 
the $i$-th possible genotype,
and each $P_{ij}$ is a given probability for an offspring to obtain the 
$i$-th genotype from the $j$-th potential father.
Let $\bar{x} = (x_1,\ldots, x_{n-1})$.
With the substitution $x_n = 1-\mathbf{1}_{n-1}^T\bar{x}$, 
we can reformulate \eqref{eq:offspring} into 
\begin{equation}\label{eq:offspring_updated}
\left\{\begin{array}{cl}
\max\limits_{\bar{x}\in\re^{n-1}} & \sum\limits_{i=1}^t N_i\log\Big(
P_{in}+\sum\limits_{j=1}^{n-1}(P_{ij}-P_{in})x_j\Big)\\
\st &  \bar{x} = (x_1,\ldots, x_{n-1})\ge 0,\\
	& 1-\mathbf{1}_{n-1}^T\bar{x}\ge 0.
\end{array}
\right.
\end{equation}
This is a log-polynomial optimization problem with simplex constraints.
Thus, it can be solved with the tighter moment relaxation \eqref{eq:relax_lme} 
using LMEs in Example~\ref{ex:lmes} (ii).
Let $N_i, P_{ij}$ be the simulation data chosen from \cite{McCulloch87}.
We solve \eqref{eq:offspring_updated} by the 
tighter moment relaxation \eqref{eq:relax_lme}.
When the relaxation is tight, we can recover the optimizer of the original 
problem \eqref{eq:offspring} by letting $x = (\bar{x}, 1-\mathbf{1}^T\bar{x})$.
For each instance, when the computed moment optimizer satisfies 
the flat truncation condition \eqref{eq:flat_truncation},
the rank of the corresponding moment matrix is one. 
By Corollary~\ref{cor:rank1}, the moment relaxation is tight
for the original problem.
The detailed simulation data and our numerical results are reported 
in Table~\ref{tab:offspring}.
In the table, we denote parameters $N = (N_1,\ldots, N_t)$ 
and $P = (P_{ij})\in\re^{t\times n}$.
The value $k$ stands for the relaxation order 
where the flat truncation holds and $r$ is the rank of the moment matrix 
associated with the moment optimizer.
The notation $x^*$ denotes the computed optimizer of the original log-polynomial
optimization problem \eqref{eq:offspring}.
The column `time' reports the elapsed time to solve the $k$-th 
order tighter moment relaxation, 
which is counted by seconds.
\begin{table}[ht]
\centering
\begin{tabular}{cccccccc}\hline
$n$ & $t$ & $N$ & $P$ & $k$ & $r$ & $x^*$ & time\\ \hline
2 & 2 &$\left[\begin{array}{r}
	77 \\  23
\end{array}\right]$ & $\begin{bmatrix}
	0.5 & 1 \\
	0.5 & 0
\end{bmatrix}$ & 2  & 1 & $\left[\begin{array}{r}
	0.4600  \\  0.5400
\end{array}\right]$  & 1.01\\ \hline 
2 & 2 &$\left[\begin{array}{r}
	63  \\  37
\end{array}\right]$ & $\begin{bmatrix}
	0.5 & 1 \\
	0.5 & 0
\end{bmatrix}$ & 2  & 1 &  $\left[\begin{array}{r}
	0.7400  \\  0.2600
\end{array}\right]$ & 1.35\\ \hline
2 & 3 &$\left[\begin{array}{r}
	49  \\  40  \\  11
\end{array}\right]$ &  $\begin{bmatrix}
	0.5 & 0.875 \\
	0.25 & 0.125 \\
	0.25 & 0
\end{bmatrix}$ & 3 & 1 & $\left[\begin{array}{r}
	0.8813  \\  0.1187
\end{array}\right]$ & 1.40\\ \hline
2 & 3 &$\left[\begin{array}{c}
	83  \\   2 \\   15
\end{array}\right]$ & $\begin{bmatrix}
	0.5 & 0.25\\
	0.5 & 0.5\\
	0 & 0.25
\end{bmatrix}$ & 2  & 1 & $\left[\begin{array}{r}
	0.6939  \\  0.3061
\end{array}\right]$ & 1.61\\ \hline
2 & 3 &$\left[\begin{array}{r}
	63 \\  17\\   20
\end{array}\right]$ & $\begin{bmatrix}
	0.5 & 0.25 \\
	0.5 & 0.5 \\
	0 & 0.25
\end{bmatrix}$ & 2  &1 & $\left[\begin{array}{r}
	0.5181  \\  0.4819
\end{array}\right]$ & 1.10\\ \hline
2 & 4 &$\left[\begin{array}{c}
	59  \\   8  \\  16 \\   17
\end{array}\right]$ & $\begin{bmatrix}
	0.25 & 0.5  \\
	0.25 & 0.5  \\
	0.25 & 0 \\
	0.25 & 0
\end{bmatrix}$ & 3  & 1 & $\left[\begin{array}{r}
	0.6599  \\  0.3401
\end{array}\right]$ & 1.03\\ \hline
2 & 5 &$\left[\begin{array}{c}
	7   \\  9  \\   4  \\  33 \\   47
\end{array}\right]$ & $\begin{bmatrix}
	0.25 & 0  \\
	0.25 & 0  \\
	0.25 & 0 \\
	0.25 & 0.5 \\
	0 & 0.5
\end{bmatrix}$ & 4  & 1 &  $\left[\begin{array}{r}
	0.2463  \\  0.7537
\end{array}\right]$ & 1.49\\ \hline
3 & 3 & $\left[\begin{array}{r}
	39 \\ 38 \\ 23
\end{array}\right]$ & $\begin{bmatrix}
	0.5 & 0 & 0.875 \\
	0.25 & 0.75 & 0.125 \\
	0.25 & 0.25 & 0
\end{bmatrix}$ & 3 & 1 & $\left[\begin{array}{r}
	0.6400  \\  0.2800  \\  0.0800
\end{array}\right]$ & 1.18\\ \hline
3 & 4 &$\left[\begin{array}{r}
	29  \\  21 \\   88\\    62
\end{array}\right]$ & $\begin{bmatrix}
	0.5 & 0 & 0 \\
	0.25 & 0.75 & 0.25 \\
	0.25 & 0.25 & 0.5 \\
	0 & 0 & 0.25
\end{bmatrix}$ & 3  & 1 & $\left[\begin{array}{r}
	0.2240 \\   0.0000 \\   0.7760
\end{array}\right]$ & 1.14\\ \hline
\end{tabular}
\caption{Numerical results for Example \ref{ex:offspring}}
\label{tab:offspring}
\end{table}
\end{example}

\begin{example}[Latent class models \cite{ChenLCM20}]
Latent class models (LCMs) \cite{McCulloch87} are a type of finite 
mixture model used to identify hidden, unobserved categorical groups 
within a population based on observed categorical data.
Let $y = (y_1,\ldots, y_d)$ be a vector of categorical random 
variables, where each $y_j$ takes values from $c_j$ categories, 
say, $\{1,\ldots, c_j\}$.
Assume each $y_j$ is independently distributed 
and its distribution ranges from $T$ subgroups,
which is determined by an indicator $s\in \{1,\ldots, T\}$. 
Let $I(\cdot)$ denote the Iverson bracket function
such that $I(P) = 1$ if $P$ is true and zero otherwise.
Then the probability density function of $y$ in the $t$-th 
subgroup is
\[
\prod\limits_{j=1}^d\prod\limits_{l=1}^{c_j} \pi_{t,j,l}^{I(y_j=l)}
\quad \mbox{with}\quad \pi_{t,j,l} = \mbox{Pr}(y_j = l|s=t),
\]
where $\mbox{Pr}(\cdot)$ is the probability notation.
The overall mixture density of latent class model is then a weighted sum
\[
p(y|\eta, \pi) = \sum\limits_{t=1}^T
\Big(\eta_t	\prod\limits_{j=1}^d\prod\limits_{l=1}^{c_j} 
\pi_{t,j,l}^{I(y_j=l)}\Big),
\]
where each $\eta_t\ge0$ denotes the weight of $t$-th subgroup that sums up to one. 
Consider $N$ samples drawn from the LCM distribution, denoted as $y^{(1)}, \ldots, y^{(N)}$.
Then the MLE for LCM is formulated as
\begin{equation}\label{eq:LCM}
\left\{\begin{array}{cl}
	\max\limits_{\eta, \pi} & \sum\limits_{i=1}^N \log \left(\sum\limits_{t=1}^T \eta_t	\prod\limits_{j=1}^d\prod\limits_{l=1}^{c_j} \pi_{t,j,l}^{I(y_j^{(i)}=l)}\right)\\
	\st & \eta = (\eta_1,\ldots, \eta_T),\\
	&  \pi = (\pi_{t,j,l})_{t\in [T], j\in[d], l\in[c_j]},\\
    & \eta\ge 0,\, \pi\ge 0,\, \sum\limits_{t=1}^T\eta_t = 1,\\
	& \sum\limits_{l=1}^{c_j} \pi_{t,j,l} = 1\, ( t\in[T],\, j\in[d]).
\end{array}
\right.
\end{equation}
We consider three instances using the same settings as in \cite{ChenLCM20}.
For each instance, we simulate $N=500$ samples for ten times 
from the corresponding LCM and approximate \eqref{eq:LCM} 
by its standard moment relaxation \eqref{eq:relax_log} at the initial order.
The resulting approximation quality is assessed by 
comparing it to the log-likelihood of the true parameters 
which is reported in \cite{ChenLCM20}.
Thus, the approximation quality depends on the chosen moment 
relaxation order and the sampling data.
The best numerical results are reported as follows.

(i) For $d=1, T = 2$, suppose each $y_i$ is a binary variable, 
i.e., $y_i\in\{0,1\}$. Let
\[\begin{array}{llll}
\eta_1 = 0.5, &  \eta_2 = 0.5, &  \pi_{1,1,1} = 0.4, &  \pi_{1,1,2} = 0.6,\\
\pi_{2,1,1} = 0.8, & \pi_{2,1,2} = 0.2.
\end{array}\] 
In our numerical experiment, it took around 0.78 second for us to obtain 
the moment upper bound $f_{\operatorname{mom},2} = -335.2519$.
This result is very close to the true likelihood $f^* = -335.29$.

(ii) For $d=1, T=3$, suppose each $y_i$ is a binary variable, 
i.e., $y_i\in\{0,1\}$. Let
\[
\begin{array}{lllll}
	\eta_1 = 0.5, & \eta_2 = 0.3, & \eta_3 = 0.2, & \pi_{1,1,1} = 0.4, &  \pi_{1,1,2} = 0.6,\\
	\pi_{2,1,1} = 0.8, & \pi_{2,1,2} = 0.2 & \pi_{3,1,1} = 0.1, & \pi_{3,1,2} = 0.9.
\end{array}
\]
In our numerical experiment, it took around 0.57 second for us to obtain 
the moment upper bound $f_{\operatorname{mom},2} = -344.45$.
This result is very close to the true log-likelihood $f^* = -344.49$.

(iii) For $d=2, T=2$, suppose each $y_i$ is a binary variable, 
i.e., $y_i\in\{0,1\}$. Let
\[
\begin{array}{llllll}
	\eta_1 = 0.5, & \eta_2 = 0.5,  & \pi_{1,1,1} = 0.4, &  \pi_{1,1,2} = 0.6,
	& \pi_{2,1,1} = 0.8,\\
\pi_{2,1,2} = 0.2 & \pi_{1,2,1} = 0.1, & \pi_{1,2,2} = 0.9 & \pi_{2,2,1} = 0.6, & \pi_{2,2,2} = 0.4.
\end{array}
\]
In our numerical experiment, it took around 0.53 second for us 
to obtain the moment upper bound $f_{\operatorname{mom},2} = -657.79$.
This result is close to the true log-likelihood $f^* = -661.30$. 
\end{example}

\section{Conclusion and Discussions}\label{sec:con}

This work explores log-polynomial optimization problems, 
characterized by logarithmic polynomial objectives with 
polynomial equality and inequality constraints. 
Using the truncated $K$-moment problems, 
we derive a hierarchy of moment relaxations
for the original non-convex formulation. 
Under specific conditions, we show that the relaxation is tight 
and provides exact global optimizer(s). 
We further establish criteria for detecting tightness 
and for extracting global solutions when they exist.
In addition, we introduce the technique of Lagrange 
multiplier expressions to construct tighter relaxations 
for log-polynomial optimization with box, simplex and
ball constraints. 
The relations among these two kinds of moment relaxations
and the original log-polynomial optimization are
summarized in Proposition~\ref{prop:rel}.

The proposed approach has practical relevance in applications 
such as maximum likelihood estimation and broader areas involving 
entropy-related optimization. 
It is interesting future work to address the scalability of this
moment relaxation approach. 
A promising way is to exploit sparsity and structure within the moment matrices. 
Another potential work is to extend the framework for non-polynomial 
or dynamic constraints.

\medskip
\noindent
{\bf Acknowledgement}
Jiawang Nie is partially supported by the NSF grant DMS- 2513254,
and Xindong Tang is partially supported by HKBU-15303423 and HKBU-22304125.


\begin{thebibliography}{100}



\bibitem{AAA13}
{\sc A.~A.~Ahmadi, A.~Olshevsky, P.~A.~Parrilo, and J.~N.~Tsitsiklis},
{\em NP-hardness of deciding convexity of quartic polynomials and related problems},
Math. Program., 137, 453--476, 2013.

\bibitem{Anari24}
{\sc N.~Anari, K.~Liu, S.O.~Gharan and C.~Vinzant},
{\em Log-concave polynomials II: High-dimensional walks and an FPRAS for counting bases of a matroid},
Ann. of Math. (2) 199 (1) 259 - 299, 2024. 

\bibitem{mosek}
{\sc MOSEK ApS}, {\em The MOSEK Optimization Toolbox for MATLAB Manual}, Version  9.0, 2019.

\bibitem{Banerjee2008}
{\sc O. Banerjee, L. El Ghaoui, and A. d’Aspremont},
{\em Model selection through sparse maximum likelihood estimation for multivariate Gaussian or binary data},
Journal of Machine Learning Research, 9,
pp. 485--516, 2008.


\bibitem{Bishop2006}
{\sc C. M. Bishop},
{\em Pattern recognition and machine learning. Information Science and Statistics}, Springer, New York, 2006.

\bibitem{Carpenter2018}
{\sc T. Carpenter, I. Diakonikolas, A. Sidiropoulos, and A. Stewart},
{\em Near-optimal sample complexity bounds for maximum likelihood estimation of multivariate log-concave densities}, In Proceedings of the 31st Conference on Learning Theory (COLT), 75, pp. 1234--1262, 2018.

\bibitem{Choi2024}
{\sc J. Choi, J. Nie, X. Tang, and S. Zhong}, {\em Generalized Nash equilibrium problems with quasi-linear constraints}, Preprint, \url{arXiv:2405.03926}, 2024.



\bibitem{ChenLCM20}
{\sc H.~Chen, L.~Han, and A.~Lim},
{\em Beyond the EM algorithm: constrained optimization methods for latent class model}
Communications in Statistics-Simulation and Computation, 51, 5222-5244, 2020.

\bibitem{Dempster1977}
{\sc A. P. Dempster, N. M. Laird, and D. B. Rubin},
{\em Maximum likelihood from incomplete data via the EM algorithm},
Journal of the Royal Statistical Society. Series B (Methodological), 39(1), pp. 1--38, 1977.

\bibitem{Fisher1922}
{\sc R. A. Fisher},
{\em On the mathematical foundations of theoretical statistics},
Philosophical Transactions of the Royal Society of London. Series A, Containing Papers of a Mathematical or Physical Character, 222, pp. 309--368, 1922.

\bibitem{Goodfellow2016}
{\sc I. J. Goodfellow, Y. Bengio, and A. Courville},
{\em Deep learning}, MIT Press, Cambridge, MA, 2016.

\bibitem{Goodfellow2014}
{\sc I. J. Goodfellow, J. Pouget-Abadie, M. Mirza, B. Xu, D. Warde-Farley, S. Ozair, A. Courville, and Y. Bengio},
{\em Generative adversarial nets}, In Proceedings of the 27th International Conference on Neural Information Processing Systems (NeurIPS), 2, pp. 2672--2680, 2014.



\bibitem{Gusmao2024}
{\sc G. S. Gusmão and A. J. Medford},
{\em Maximum-likelihood estimators in physics-informed neural networks for high-dimensional inverse problems}, Computers and Chemical Engineering, 181, 108547, 2024

\bibitem{Hansen2024}
{\sc D. Hansen, D. C. Maddix, S. Alizadeh, G. Gupta, and M. W. Mahoney},
{\em Learning physical models that can respect conservation laws},
Physica D: Nonlinear Phenomena, 457, 133952, 2024



\bibitem{Karanam2025}
{\sc B. A. Karanam, S. Mathur, S. Sidheekh, and S. Natarajan},
{\em A Unified Framework for Human-Allied Learning of Probabilistic Circuits}, In Proceedings of the AAAI Conference on Artificial Intelligence, 39(17), pp. 17779--17787, 2025.

\bibitem{Vinayak2019}
{\sc R. Korlakai Vinayak, W. Kong, G. Valiant, and S. M. Kakade},
{\em Maximum likelihood estimation for learning populations of parameters}, In Proceedings of the 36th International Conference on Machine Learning (ICML), 97, pp.6448--6457, 2019.



\bibitem{lasserre2001}
{\sc J.B. Lasserre}, {\em Global optimization with polynomials and the problem of moments}, SIAM Journal on Optimization, 11(3), pp. 796–817, 2001.

\bibitem{Lauritzen2019}
{\sc S. Lauritzen, C. Uhler, and P. Zwiernik},
{\em Maximum Likelihood Estimation in Gaussian models under total positivity},
The Annals of Statistics, 4(4), pp. 1835--1863, 2019.



\bibitem{Yalmip}
{\sc J. Löfberg}, {\em YALMIP: A Toolbox for Modeling and Optimization in MATLAB}, In Proceedings of the CACSD Conference, Taipei, Taiwan, 2004

\bibitem{Mannor2005}
{\sc S. Mannor, D. Peleg, and R. Rubinstein},
{\em The cross entropy method for classification},
In Proceedings of the 22nd International Conference on Machine Learning (ICML), pp. 561–-568, 2005.

\bibitem{Mao2023}
{\sc A. Mao, M. Mohri, and Y. Zhong},
{\em Cross-entropy loss functions: theoretical analysis and applications}, In Proceedings of the 40th International Conference on Machine Learning (ICML), 202, pp. 23803--23828, 2023.


\bibitem{McCulloch87}
{\sc C. E. McCulloch}, {\em Maximum likelihood estimation in a multinomial mixture model}, Technical Report BU-934-MA, Biometrics Unit, Cornell University, Ithaca, NY, 1987.

\bibitem{EMBook}
{\sc G. J. McLachlan and T. Krishnan}, {\em The EM Algorithm and Extensions}, Wiley, New York, 1997.



\bibitem{MPO}
{\sc J.~Nie},
{\em Moment and Polynomial Optimization},
SIAM, Philadelphia, 2023.

	

\bibitem{nie2013certifying}
\leavevmode\vrule height 2pt depth -1.6pt width 23pt,
{\em Certifying convergence of Lasserre’s hierarchy via flat truncation},
Mathematical Programming, 142, pp.~485--510, 2013.

\bibitem{nie2014optimality}
\leavevmode\vrule height 2pt depth -1.6pt width 23pt, {\em Optimality
  conditions and finite convergence of {L}asserre's hierarchy}, Mathematical
  Programming, 146, pp.~97--121, 2014.
 			


\bibitem{Nie2019}
\leavevmode\vrule height 2pt depth -1.6pt width 23pt, {\em Tight relaxations for polynomial optimization and {L}agrange multiplier expressions}, Mathematical Programming, 178, pp.~1--37, 2019.

\bibitem{NieSunTangZ}
{\sc J. Nie, D. Sun, X. Tang, and M. Zhang}, {\em Solving polynomial variational inequality problems via Lagrange multiplier expressions and Moment-SOS relaxations}, Preprint, \url{arXiv:2303.12036}, 2023.

\bibitem{nie2023convex}
{\sc J.~Nie and X.~Tang}, {\em Convex generalized {N}ash equilibrium problems
  and polynomial optimization}, Mathematical Programming, 198,
  pp.~1485--1518, 2023.

\bibitem{NieTangZgnep21}
{\sc J.~Nie, X.~Tang, and S.~Zhong}, {\em Rational generalized Nash
  equilibrium problems}, SIAM Journal on Optimization, 33, pp. 1587--1620, 2023.

\bibitem{NieWangZBilevel21}
{\sc J.~Nie, L.~Wang, J.J.~Ye, and S.~Zhong}, {\em A Lagrange multiplier expression method for bilevel polynomial optimization}, SIAM Journal on Optimization, 31, pp. 2368-2395, 2021.




\bibitem{Putinar}
{\sc M.~Putinar}, {\em Positive polynomials on compact semi-algebraic sets},
  Indiana University Mathematics Journal, 42, pp.~969--984, 1993.

\bibitem{Rayas2022}
{\sc A. Rayas, R. Anguluri, and G. Dasarathy},
{\em Learning the Structure of Large Networked Systems Obeying Conservation Laws}, In Proceedings of the 36th Annual Conference on Neural Information Processing Systems (NeurIPS 2022), 35, pp. 14637-–14650, 2022.

\bibitem{Silverman1986}
{\sc B. W. Silverman},
{\em Density estimation for statistics and data analysis}, Chapman and Hall, 1986.


\bibitem{Srivastava2024}
{\sc A. Srivastava, A. Bayati and S. M. Salapaka},
{\em Sparse Linear Regression with Constraints: A Flexible Entropy-Based Framework}, In Proceedings of the 2024 European Control Conference (ECC), pp. 2105--2110, 2024.

\bibitem{Zhang2018}
{\sc Z. Zhang and M. R. Sabuncu},
{\em Generalized cross entropy loss for training deep neural networks with noisy labels},
In Proceedings of the 32nd International Conference on Neural Information Processing Systems (NeurIPS), pp. 8792--8802, 2018.



\end{thebibliography}
\end{document}